\documentclass[10pt, reqno, a4paper]{amsart}

\usepackage[T1]{fontenc}
\usepackage[english]{babel}
\usepackage[dvipsnames,table,xcdraw]{xcolor}
\definecolor{myred}{RGB}{228,26,28}
\definecolor{myblue}{RGB}{55,124,184}
\definecolor{mygreen}{RGB}{77,175,74}
\usepackage[colorlinks=true,urlcolor=myblue,linkcolor=myblue,runcolor=myred,citecolor=myred,linktoc=all]{hyperref}

\usepackage{geometry}
\usepackage{graphicx}
\usepackage{array}
\usepackage{wrapfig}
\usepackage{enumitem}
\setlist[itemize]{label={$\vcenter{\hbox{\tiny$\bullet$}}$}}
\usepackage{appendix}
\usepackage{caption}
\usepackage{subcaption}
\usepackage{csquotes}

\usepackage{amsmath,mathtools,amssymb,amsthm,amsfonts}
\usepackage[foot]{amsaddr}
\usepackage{bm}
\usepackage{esint}
\usepackage{chemformula}

\renewcommand{\lq}{\leqslant}
\newcommand{\gq}{\geqslant}
\newcommand{\set}[1]{\left\{ #1 \right\}}

\newcommand{\Rb}{\mathbb{R}}
\newcommand{\Cb}{\mathbb{C}}
\newcommand{\Zb}{\mathbb{Z}}
\newcommand{\Forall}{\forall\ }
\newcommand{\Exist}{\exists\ }

\renewcommand{\d}{{\rm d}}
\newcommand{\prt}[1]{\left( #1 \right)}
\newcommand{\e}{{\rm e}}
\renewcommand{\i}{{\rm i}}

\newcommand{\per}{{\rm per}}

\newcommand{\norm}[1]{\left\| #1 \right\|}
\newcommand{\abs}[1]{\left| #1 \right|}

\newcommand{\Hc}{\mathcal{H}}

\newcommand{\Sc}{\mathcal{S}}

\newcommand{\Lc}{\mathcal{L}}

\newcommand{\C}{\mathbb{C}}
\newcommand{\N}{\mathbb{N}}
\newcommand{\R}{\mathbb{R}}
\newcommand{\Z}{\mathbb{Z}}

\newcommand{\LL}{\mathbb{L}}

\newcommand{\wt}{\widetilde}
\newcommand{\wh}{\widehat}
\newcommand{\wb}{\overline}

\newcommand{\Ecut}{E_{\rm cut}}

\theoremstyle{plain}
\newtheorem{proposition}{Proposition}

\theoremstyle{definition}
\newtheorem{remark}{Remark}

\theoremstyle{plain}
\newtheorem{theorem}{Theorem}

\theoremstyle{plain}
\newtheorem{lemma}{Lemma}

\title[Periodic Schr\"odinger equations with analytic potentials]{A priori error analysis of linear and nonlinear periodic Schr\"odinger equations with analytic potentials}

\usepackage[foot]{amsaddr}
\author{Eric Canc\`es$^1$}
\author{Gaspard Kemlin$^{2}$}
\author{Antoine Levitt$^{3}$}

\address{$^1$CERMICS, \'Ecole des Ponts and Inria Paris, 6 \& 8 avenue Blaise Pascal, 77455 Marne-la-Vall\'ee, France}
\address{$^2$CERMICS, \'Ecole des Ponts and Inria Paris, 6 \& 8 avenue Blaise
  Pascal, 77455 Marne-la-Vall\'ee, France, now at LAMFA, Universit\'e de
  Picardie Jules Verne UFR des Sciences, 33 rue Saint-Leu, 80039 Amiens, France}
\address{$^3$LMO, Université Paris-Saclay, 307 rue Michel Magat, 91400 Orsay, France}
\address{emails: \texttt{eric.cances@enpc.fr},
  \texttt{gaspard.kemlin@u-picardie.fr},
  \texttt{antoine.levitt@universite-paris-saclay.fr}}

\begin{document}

\begin{abstract}
  This paper is concerned with the numerical analysis of linear and
  nonlinear Schrödinger equations with periodic analytic potentials. We prove
  that, for linear equations, when the potential is
  analytic in a strip of width $A$ of the complex plane, the solution is
  analytic in the same strip, ensuring an exponential convergence of the
  planewave discretization of the equation with rate $A$. On the other
  hand, for nonlinear equations, we find that the solution may be analytic only
  in a strip of width smaller than $A$. This behavior is illustrated by two
  examples using a combination of numerical and analytical arguments.
\end{abstract}

\maketitle

\section{Introduction}
In this paper, we study models inspired by Kohn--Sham Density
Functional Theory (KS-DFT). KS-DFT is currently the most popular
model in quantum chemistry and materials science as it offers a good
compromise between accuracy and computational efficiency. It aims at
computing, for a given configuration of the nuclei of the molecular system or
material of interest, the electronic ground-state energy and density. From the
latter, it is possible to compute the effective forces acting on the nuclei in
this configuration, and thus to identify the (meta)stable equilibrium
configurations of the system, or to simulate the dynamics of the molecular
system in various thermodynamic conditions. In materials science applications,
computations are commonly done  in a periodic simulation cell, which can be either the
unit cell of a crystal (for the special case of perfect crystals), or a
supercell (for all the other cases: crystals with defects, disordered alloys,
glassy materials, liquids...).

\medskip

We denote by $\LL=\Z \bm a_1 + \Z \bm a_2 + \Z \bm a_3$ the periodic lattice,
where $(\bm a_1,\bm a_2, \bm a_3)$ is a non-necessarily orthonormal basis of
$\R^3$, and by $\Omega=[0,1) \bm a_1 + [0,1) \bm a_2 + [0,1) \bm a_3$ the
simulation cell. Let us denote by
$$
L^2_{\per,\LL}\coloneqq \{u \in L^2_{\rm loc}(\R^3,\C) \; | \; u \; \text{ is
} \LL\mbox{-periodic} \}
$$
the Hilbert space of complex-valued $\LL$-periodic locally square integrable
functions on $\R^3$, endowed with its usual inner product.
The KS-DFT equations read
\begin{equation}\label{eq:KS-DFT}
  H_\rho \varphi_i = \lambda_i \varphi_i, \quad
  \langle   \varphi_i,\varphi_{j}\rangle_{L^2_{\per,\LL}} = \delta_{ij}, \quad \rho(\bm x)
  = 2\sum_{i=1}^{N_{\rm p}} |\varphi_i(\bm x)|^2,
\end{equation}
where $H_\rho$ is the Kohn--Sham Hamiltonian, a self-adjoint operator on
$L^2_{\per,\LL}$ bounded below and with compact resolvent.   The
$\varphi_i$'s are the Kohn--Sham orbitals, and the $\lambda_i$'s their
energies. Since $H_\rho$ depends on $\rho$, which in turn depends on the
eigenfunctions $\varphi_i$, $1 \lq i \lq N_{\rm p}$, \eqref{eq:KS-DFT} is a
nonlinear eigenproblem. The parameter $N_{\rm p}$ represents physically the
number of valence electron pairs per simulation cell and $\rho$ the
ground-state electronic density. We assume here, as is the case for most
physical systems, that
$\lambda_1 \lq \lambda_2 \lq \cdots \lq \lambda_{N_{\rm p}}$ are the lowest
$N_{\rm p}$ eigenvalues of $H_\rho$ ({\it Aufbau} principle). The Kohn--Sham
Hamiltonian with pseudopotentials reads
$$
H_\rho = - \frac 1 2 \Delta + V_{\rm nl} + V_{{\rm loc},\rho}^{\rm Hxc}
$$
where $V_{\rm nl}$ is a finite-rank self-adjoint operator (the nonlocal part
of the pseudopotential), and
$$
V_{{\rm loc},\rho}^{\rm Hxc}(\bm x) = V_{\rm loc}(\bm x) + V_{{\rm
    H},\rho}(\bm x) + V_{{\rm xc},\rho}(\bm x)
$$
is a periodic real-valued function depending (nonlocally) on $\rho$. The function
$V_{\rm loc}$ is the local component of the pseudopotential, the Hartree
potential $V_{{\rm H},\rho}$ is the unique solution with zero mean to the
periodic Poisson equation
$$
- \Delta V_{{\rm H},\rho}(\bm x) = 4\pi \left( \rho(\bm x) -
  \frac{1}{|\Omega|}\int_\Omega \rho \right), \qquad \int_\Omega V_{{\rm
    H},\rho} = 0,
$$
and the function $V_{{\rm xc},\rho}$, called the exchange-correlation
potential, depends on the chosen approximation of the exchange-correlation
energy functional. In the simple X$\alpha$ model \cite{slaterSimplificationHartreeFockMethod1951},
$V_{{\rm xc},\rho}(\bm x) = - C_{\rm D} \rho(\bm x)^{1/3}$, where
$C_{\rm D} > 0$ is the Dirac constant.

\medskip

It is not mandatory to use pseudopotentials in KS-DFT calculations. Some
softwares allow for all-electrons calculations in which the total
pseudopotential operator $V_{\rm nl}+V_{\rm loc}$ is replaced with a local
potential with Coulomb singularities at the positions of the nuclei. However,
most calculations are done with pseudopotentials, or use the formally similar
Projector Augmented Wave (PAW) method \cite{blochlProjectorAugmentedwaveMethod1994},
for three reasons: (i) core
electrons are barely affected by the chemical environment and can
usually be considered to occupy ``frozen states'', (ii) in heavy atoms,
core electrons must be dealt with relativistic quantum models which makes the
simulation more expensive from a computational viewpoint, (iii)
due to the Coulomb singularities, all-electron Kohn--Sham orbitals have cusps
at the positions of the nuclei and are therefore only Lipschitz continuous,
while the Kohn--Sham orbitals computed with pseudopotentials are much more
regular and can be well approximated with Fourier spectral methods (usually
called planewave discretization methods in the field).

\medskip

Several methods for constructing pseudopotentials have been proposed in the
literature, leading to local and nonlocal functions of different regularity.
As expected, the
rate of convergence of the planewave discretization method is directly linked
to the regularity of these functions. The {\it a priori} error analysis of
this problem was performed in~\cite{cancesNumericalAnalysisplanewave2012} for
pseudopotential with Sobolev regularity. It was proved in particular, for the
simple X$\alpha$ exchange-correlation functional, but also for the much more
popular local density approximation (LDA) exchange-correlation functional,
that if the local and nonlocal parts of the pseudopotential are
in the periodic Sobolev space of order $s> 3/2$, then the Kohn--Sham orbitals
$\varphi_i$ and the density $\rho$ are in the periodic Sobolev space of order
$s+2$, and (optimal) polynomial convergence rates were obtained in any Sobolev
spaces of order $r$ with $-s <  r < s+2$. In addition, as for linear
second-order elliptic eigenproblems, the error on the eigenvalues converges to
zero as the square of the error on the eigenfunctions evaluated in $H^1$-norm.
The analysis in~\cite{cancesNumericalAnalysisplanewave2012} covers for example
the case of Troullier--Martins pseudopotentials
\cite{troullierEfficientPseudopotentialsPlanewave1991}, for which
$s = \frac 72-\varepsilon$ for all $\varepsilon > 0$. On the other hand, these estimates are not sharp in the case of
Goedecker--Teter--Hutter (GTH)
pseudopotentials
\cite{goedeckerSeparableDualspaceGaussian1996,hartwigsenRelativisticSeparableDualspace1998},
for which the local and nonlocal contributions
are periodic sums of Gaussian-polynomial functions, and therefore have entire
continuations to the whole complex plane. Such pseudopotentials are
implemented in different DFT softwares, such as BigDFT
\cite{ratcliffFlexibilitiesWaveletsComputational2020}, Quantum Espresso
\cite{giannozziQUANTUMESPRESSOModular2009} or Abinit
\cite{gonzeAbinitProjectImpact2020,romeroABINITOverviewFocus2020},
as well as DFTK, a recent electronic
structure package in the \texttt{Julia}
language~\cite{herbstDFTKJulianApproach2021}.

\medskip

The purpose of this paper is to investigate this case.
While it has been known for a long time (see e.g.
\cite{bernsteinNatureAnalytiqueSolutions1904,friedmanRegularitySolutionsNonLinear1958,morreyAnalyticitySolutionsAnalytic1958,morreyAnalyticitySolutionsAnalytic1958b,petrovskiiAnalyticiteSolutionsSystemes1939}
and references therein for historical insight or
\cite{blattAnalyticitySolutionsNonlinear2020,hashimotoRemarkAnalyticitySolutions2006}
for more recent developments) that the solutions to elliptic equations on $\R^d$
with real-analytic data have an analytic continuation in a complex
neighborhood of $\R^d$, the size of this neighborhood is {\it a priori}
unknown. In the periodic case we are considering, the latter directly impact
the decay rate of the Fourier coefficients of the solution, hence the
convergence rate of the planewave discretization method. For pedagogical
reasons, we will work most of the time with one dimensional linear or nonlinear
Schr\"odinger equations, because (i) it is easier to visualize analytic or
entire continuations of functions originally defined on the real space $\R^d$
when $d=1$, and (ii) exponential convergence rates of planewave discretization
methods are easier to spot in 1D. However, most of our arguments
extend to the multidimensional case. In Section~\ref{sec:analytic}, we introduce a
hierarchy of spaces $(\Hc_A)_{A > 0}$ of complex-valued $2\pi$-periodic
functions on the real line having analytic continuations to the strip
$\R + \i (-A,A)$ of the complex plane.
We then pick a real-valued function $V \in \Hc_B$ for some
$B > 0$ and consider the one-dimensional Schr\"odinger operator $H=-\Delta + V$.
A low \emph{vs} high-frequency decomposition of the periodic $L^2$ space allows to
prove that for all $0 < A < B$, the solution $u$ to the linear equation $Hu=f$
is in $\Hc_A$ whenever $f \in \Hc_A$ (see Section~\ref{sec:lin_pb}), and that
the eigenfunctions of $H$ are in $\Hc_A$ (see Section~\ref{sec:lin_egval}). We
rely on this result to prove in Section~\ref{sec:numeric} that the planewave
discretization method converges exponentially fast in this case and we provide
a numerical illustration of these results.
We turn in Section~\ref{sec:nonlin} to the nonlinear setting,
where we present two examples for which we show, using a combination of
numerical and analytical tools, that the analyticity strip of the solution may
be much smaller than the one of the data. Finally, we consider in the Appendix the
multidimensional case, which is an immediate extension, and its application to
Kohn--Sham models.

\section{Spaces of analytic functions}
\label{sec:analytic}

Let us first introduce some notation. We denote by $L^2_\per(\Rb,\Cb)$ the
space of locally square integrable complex-valued $2\pi$-periodic functions on $\Rb$,
endowed with its natural inner product
$$
(u,v)_{L^2_\per}\coloneqq\int_0^{2\pi} \overline{u(x)} \, v(x) \, dx,
$$
and by $\Sc'_\per(\R,\C)$ the space of tempered complex-valued $2\pi$-periodic
distributions on $\R$. For each $u \in \Sc'_\per(\R,\C)$, we denote by
$(\widehat u_k)_{k \in \Z}$ the Fourier coefficients of $u$  with the
following normalization convention:
$$
\Forall u \in L^2_\per(\Rb,\Cb), \quad \Forall k \in \Z, \quad
\wh{u}_k \coloneqq (e_k,u)_{L^2_\per} = \frac{1}{\sqrt{2\pi}} \int_0^{2\pi}
u(x)\e^{-\i k x}\d x,
$$
where $e_k(x)\coloneqq \frac{1}{\sqrt{2\pi}}\e^{\i kx}$ is the $L^2_\per$-normalized
Fourier mode with wavevector $k \in \Z$. Recall that the $2\pi$-periodic
Sobolev spaces are the Hilbert spaces $H^s_\per(\Rb,\Cb)$, $s \in \R$, defined by
\begin{equation*}
  H^s_\per(\Rb,\Cb) \coloneqq \displaystyle{\set{
      u \in L^2_\per(\Rb,\Cb) \;\middle|\;
      \sum_{k\in\Zb}(1 + |k|^2)^s  \abs{\wh{u}_k}^2  < \infty
    }},
\end{equation*}
\begin{equation*}
  (u,v)_{H^s_\per}\coloneqq\sum_{k \in \Z}  (1 + |k|^2)^s \, \overline{\wh{u}_k} \, \wh{v}_k.
\end{equation*}
We will also use the self-explanatory notation $C^k_\per(\R,\R)$,
$C^k_\per(\R,\C)$, $L^p_\per(\R,\R)$, $L^p_\per(\R,\C)$ for
$k \in \N \cup \{\infty\}$ and $1 \lq p \lq \infty$, all these spaces being
endowed with their natural norms or topologies.  We now introduce, for any
$A>0$, the (Hardy-like) space
\[
  \Hc_A \coloneqq \displaystyle{\set{
      u \in L^2_\per(\Rb,\Cb) \;\middle|\;
      \sum_{k\in\Zb} w_A(k) \abs{\wh{u}_k}^2  < \infty   }}
\]
where
\[
  \quad w_A(k) \coloneqq \cosh(2Ak) = \tfrac 1 2 (\e^{2Ak}+\e^{-2Ak}),
\]
endowed with the inner product
\[
  (u,v)_{A} \coloneqq \sum_{k\in\Zb} w_A(k) \, \overline{\wh{u}_k} \, \wh{v}_k.
\]
Note that $\Hc_A$ can be canonically identified with the space of analytic functions
$$
\wt{\Hc}_A \coloneqq \set{u : \Omega_A \to \C
  \text{\normalfont~analytic}  \;\middle|\;
  \begin{matrix}
    [-A,A] \ni y\mapsto u(\cdot + \i y) \in  L^2_\per(\Rb,\Cb)
    \text{ \normalfont~continuous},\\
    \displaystyle\int_0^{2\pi}\left( \abs{u(x+\i A)}^2 + \abs{u(x-\i A)}^2 \right) \d x < \infty
  \end{matrix}},
$$
where $\Omega_A\coloneqq\Rb +\i (-A,A) \subset \C$ is the horizontal strip
of width $2A$ of the complex plane centered on the real axis,
endowed with the inner product
\[
  (u,v)_{\wt{\Hc}_A} = \frac{1}{2} \left(
    (u(\cdot+\i A),v(\cdot+\i A))_{L^2_\per} +
    (u(\cdot-\i A),v(\cdot-\i A))_{L^2_\per} \right).
\]
The canonical unitary
mapping $\Hc_A$ onto $\wt{\Hc}_A$ is the analytic continuation:
any function $u\in\Hc_A$ has a unique analytic continuation
$u : \Omega_A \to \C$ given by
$$
\Forall z = x+\i y \in \Omega_A, \quad u(z) = \sum_{k \in \Z}
\widehat u_k \, \frac{\e^{\i kz}}{\sqrt{2\pi}} = \sum_{k \in \Z} \widehat u_k
\,\e^{-ky} \, e_k(x).
$$
It can be easily seen that the Fourier coefficients of
$ u(\cdot\pm \i A)$ are the Fourier coefficients of $u$ rescaled by a factor
$\e^{\mp kA}$ and that the function
$(-A,A) \ni y \mapsto u(\cdot + \i y) = \sum_{k \in \Z} \widehat u_k \,\e^{-ky} \, e_k(\cdot)  \in L^2_\per(\R,\C)$
has a unique  continuation to $[-A,A]$ .  Therefore,
\[
  \begin{split}
    \|u\|_{\wt{\Hc}_A}^2  &=  \frac{1}{2} \left( \|
      u(\cdot+\i A)\|^2 _{L^2_\per} +  \|u(\cdot-\i A)\|^2 _{L^2_\per}
    \right) \\ &= \frac 12 \left(  \sum_{k \in \Z} \abs{\widehat u_k \,\e^{-kA}}^2 +
      \sum_{k \in \Z} \abs{\widehat u_k \,\e^{+kA}}^2  \right) \\ &=  \sum_{k \in \Z}
    w_A(k) \abs{\widehat u_k}^2 = \|u\|_{A}^2.
  \end{split}
\]

We record for future use the
\begin{proposition}
  \label{prop:Vbounded}
  Let $B > 0$.  Then, for all $0 < A < B$, the multiplication by a function
  $V_B \in \Hc_B$ defines a bounded operator on $\Hc_A$.
\end{proposition}

\begin{proof}
  The proof is immediate from the analyticity of $V$ in
  $\Omega_{B}$, which implies that
  $\sup_{z, |{\rm Im(z)| \lq A}}|V(z)| < \infty$.
\end{proof}

\section{The linear case}

\subsection{The linear elliptic problem}
\label{sec:lin_pb}

We consider in a first stage the one-dimensional linear elliptic problem
\begin{equation}
  \label{eq:linear}
  \mbox{seek } u \in H^2_\per(\R,\C) \quad \mbox{such that} \quad  -\Delta u +
  Vu = f \text{ on } \Rb,
\end{equation}
where $V \in L^2_\per(\R,\R)$ and $f \in L^2_\per(\R,\C)$ are given $2\pi$-periodic
functions. We assume in this section that $V \gq 1$.
It is then standard that \eqref{eq:linear} has a unique
solution $u$ satisfying the {\it a priori} bounds
\begin{equation} \label{eq:classical_bounds}
  \norm{u}_{L^2_\per} \lq \frac{\norm{f}_{L^2_\per}}{\alpha} \quad \mbox{and}
  \quad  \norm{u}_{H^1_\per} \lq \norm{f}_{H^{-1}_\per},
\end{equation}
where $\alpha \coloneqq \lambda_1(-\Delta+V) \gq 1$ is the smallest eigenvalue
of the self-adjoint operator $H=-\Delta + V$ on $L^2_\per(\R,\C)$. By
bootstrap arguments, $u \in H^{s+2}_\per(\R,\C)$ whenever
$V$ and $f$ are in $H^s_{\rm per}$, for any $s \gq 0$. The following result
deals with the case of real-analytic potentials $V$ and right-hand sides $f$.
\begin{theorem}
  \label{thm:lin_bounds}
  Let $B>0$ and $V\in\Hc_B$ be real-valued and such that $V \gq 1$ on $\R$.
  Then, for all $0<A<B$ and $f\in\Hc_A$, the unique
  solution $u$ of \eqref{eq:linear} is in
  $\Hc_A$.  Moreover, we have the following estimate
  \begin{equation}\label{eq:bounds_linear_eq}
    \Exist C > 0 \text{ independent of } f \text{ such that }
    \norm{u}_A \lq C\norm{f}_A.
  \end{equation}
  As a consequence, if $V$ and $f$ are entire, then so is $u$.
\end{theorem}

\begin{proof}
  For $N>0$, we consider the decomposition
  $L^2_\per(\Rb,\Cb) = X_N \oplus X_N^\perp$ where
  \begin{equation} \label{eq:space_XN}
    X_N \coloneqq \mbox{Span}(e_k, \; |k| \lq N)= \set{u\in L^2_\per(\Rb,\Cb)
      \;|\; \wh{u}_k = 0,\ \Forall |k|>N}.
  \end{equation}
  Let $\Pi_N$ be the orthogonal projector on $X_N$
  and $\Pi_N^\perp \coloneqq 1 - \Pi_N$ the orthogonal projector on
  $X_N^\perp$. Note that the restriction of $\Pi_N$ to the Sobolev space
  $H^s_\per(\R,\C)$, $s > 0$, is also the orthogonal projector on $X_N$ for
  the $H^s_\per$ inner product, and that the same property holds for the
  Hilbert spaces $\Hc_A$.

  \medskip

  For a fixed $N$, we decompose $u$ as $u = u_1 + u_2$ with $u_1\in X_N$ and
  $u_2\in X_N^\perp$.  As $u_1$ has compact Fourier support, it obviously
  belongs to $\Hc_A$ and we have
  the estimate
  \begin{equation}
    \label{eq:u1}
    \norm{u_1}_A \lq \norm{u}_{L^2_\per} \sqrt{w_A(N)} \lq
    \frac{\norm{f}_{L^2_\per}}{\alpha} \sqrt{w_A(N)}.
  \end{equation}
  Let us show that, for $N$ large enough, $u_2$ also belongs to $\Hc_A$.
  Projecting $-\Delta u + Vu = f$ onto $X_N^\perp$, we get
  \begin{equation}\label{eq:proj_eq}
    T_{22}u_2 + V_{22}u_2 = f_2 - V_{21}u_1,
  \end{equation}
  where $T_{22}$ is the restriction to the invariant subspace $\Hc_A \cap X_N^\perp$ of
  the self-adjoint operator $-\Delta$ on $L^2_\per(\R,\C)$, and where, in view
  of Proposition~\ref{prop:Vbounded},
  $V_{22} \coloneqq \Pi_N^{\perp}V\Pi_N^{\perp} \in \Lc(\Hc_A \cap X_N^\perp)$,
  $V_{21} \coloneqq \Pi_N^{\perp}V\Pi_N \in \Lc(\Hc_A \cap X_N,\Hc_A \cap X_N^\perp)$ and
  $f_2 \coloneqq \Pi_N^{\perp}f \in \Hc_A \cap X_N^\perp$. The operator
  $T_{22}$ on $\Hc_A \cap X_N^\perp$ is bounded from below by $N^2$ and is therefore invertible with
  inverse $T_{22}^{-1}$ bounded by $N^{-2}$.
  As $f_2, V_{21}u_1\in\Hc_A \cap X_N^\perp$, we can
  therefore rewrite \eqref{eq:proj_eq} as
  \begin{equation}
    (1+T_{22}^{-1}V_{22}) u_2 = T_{22}^{-1} (f_2 - V_{21}u_1).
  \end{equation}
  Since $\|T_{22}^{-1}\|_{\Lc(\Hc_A\cap X_N^\perp)} \lq N^{-2}$ and
  $\|V_{22}\|_{\Lc(\Hc_A\cap X_N^\perp)} \lq \|V\|_{\Lc(\Hc_A)}$, we can
  choose $N$ large enough so that the operator
  $(1+T_{22}^{-1}V_{22})$ is invertible on $\Hc_A\cap X_N^\perp$. It
  then holds
  \begin{equation} \label{eq:u2}
    u_2 = (1+T_{22}^{-1}V_{22})^{-1} T_{22}^{-1} (f_2 - V_{21}u_1)
  \end{equation}
  and the result follows from
  $\|V_{21}\|_{\Lc(\Hc_{A})} \lq \|V\|_{\Lc(\Hc_{A})}$ and the bound
  \eqref{eq:u1} on $u_{1}$.
\end{proof}

\medskip

\begin{remark}[Additional regularity] We mention here some cases where we can
obtain additional regularity on the solution $u$ to \eqref{eq:linear}. First,
note that, if we additionally require that
$V(\cdot \pm \i B) \in L^{\infty}_{\per}(\Rb,\Cb)$, then the same reasoning
leads to $u\in\Hc_B$ if $f\in\Hc_B$. Next, one might wonder whether there exist $0 < B < B_0 < A$, $f \in \Hc_A$ and $V \in \Hc_B$ such that
$V \notin \Hc_C$ for all $B_0 \lq  C \lq A$ but nevertheless $u\in\Hc_A$. The
following example shows that this can happen.
Consider the $2\pi$-periodic real-valued functions $V$ and $u_0$ on $\R$ defined by
\[
\Forall x\in \R,\quad V(x) = 1+\frac1{\gamma + \sin^2(x)}
\quad\text{and}\quad u_0(x) = \frac{\gamma + \sin^2(x)}{2\gamma +
\sin^2(x)}.
\]
We have that (i) $V\in\Hc_B$ if and only if $B < B_0\coloneqq{\rm
arcsinh}(\sqrt{\gamma})$ and (ii) $u_0$ and $Vu_0$ both belong to $\Hc_A$ for any $A<A_0\coloneqq{\rm arcsinh(\sqrt{2\gamma})}$ (in
particular for any $A$ such that $B_0 \lq  A < A_0$). Therefore, if
we set $f\coloneqq -\Delta u_0 + Vu_0$, then $f\in\Hc_A$ for any $A$ such that
$B_0 < A < A_0$ and the unique solution to \eqref{eq:linear} with data $V$ and $f$
is $u=u_0 \in \Hc_A$ despite $V\notin\Hc_C$ for all $B_0 \lq  C  \lq A$.
\end{remark}

\medskip

\subsection{The linear eigenvalue problem}
\label{sec:lin_egval}

We now focus on the linear eigenvalue problem,
\begin{equation}
  \label{eq:lin_egval}
  -\Delta u + Vu = \lambda u, \quad \norm{u}_{L^2_\per(\Rb,\Cb)} = 1,
\end{equation}
where $V\in\Hc_B$ for some $B>0$. Using the same technique as for the proof
of Theorem~\ref{thm:lin_bounds}, we get the
\begin{theorem}
  \label{thm:lin_egval}
  Let $B>0$, $V\in\Hc_B$ be real-valued, and $(u,\lambda) \in H^2_\per(\Rb,\Cb)\times\Rb$
  be a normalized eigenmode of $H=-\Delta + V$, with isolated eigenvalue (i.e. a solution to
  \eqref{eq:lin_egval}).
  Then, $u$ is in $\Hc_A$ for all $0 < A < B$.
  As a consequence, if $V$ is entire, then so is $u$.
\end{theorem}

\begin{proof}
  The proof proceeds along the same lines as that of Theorem~\ref{thm:lin_bounds}. Let $(u,\lambda) \in H^2_\per(\Rb,\Cb)\times\Rb$ be a solution to
  \eqref{eq:lin_egval}. Using the same notation as in the proof of
  Theorem~\ref{thm:lin_bounds}, we decompose $u$ as $u = u_1 + u_2$ with
  $u_1\in X_N$ and $u_2\in X_N^\perp$, and observe that
  for $N$ large enough,
  \[
    u_2 = - (1 + T_{22}^{-1}(V_{22} - \lambda))^{-1} T_{22}^{-1} V_{21}u_1,
  \]
  with
  \[
    \norm{T_{22}^{-1}(V_{22}-\lambda)}_{\Lc(\Hc_A \cap
      X_N^\perp)} \lq \frac{\|V\|_{\Lc(\Hc_A)} + |\lambda|}{N^2}.
  \]
  Therefore, choosing $N$ large enough, we
  have
  $$
  \norm{(1 + T_{22}^{-1}(V_{22} - \lambda))^{-1} T_{22}^{-1}}_{\Lc(\Hc_A \cap
    X_{N}^\perp)} < 1,
  $$
  from which we deduce that $u_2\in\Hc_A$ and the result follows.
\end{proof}

\subsection{Planewave approximation of the linear Schr\"odinger equation}
\label{sec:numeric}

Using $X_N = \mbox{Span}(e_k, \; |k| \lq N) \subset H^1_\per(\R,\C)$ as a
variational approximation space for \eqref{eq:lin_egval}, we obtain the
finite-dimensional problem
\begin{equation}
  \label{eq:var_approx}
  \begin{cases}
    \text{seek } (u_N,\lambda_N) \in X_N \times \Rb \text{ such that }
    \norm{u_N}_{L^2_\per(\Rb,\Cb)}=1 \text{ and }\\
    \displaystyle{\Forall v_N \in X_N,\quad \int_0^{2\pi}  \overline{\nabla
        u_N} \cdot \nabla v_N +
      \int_0^{2\pi} V \overline{u_N} v_N = \lambda_N \int_0^{2\pi}
      \overline{u_N} v_N,}
  \end{cases}
\end{equation}
which is equivalent to seeking the eigenpairs of the Hermitian matrix
$H_N$
with entries
$$
[H_N]_{kk'} \coloneqq |k|^2 \delta_{kk'} + \widehat V_{k-k'}, \quad k,k' \in
\Z, \; |k| \lq N, \; |k'| \lq N.
$$
The following theorem states that if $V \in \Hc_B$ for some $B>0$, the
planewave discretization method has an exponential convergence rate. Note that
a similar result holds for the planewave approximation of the linear problem
$-\Delta u + Vu = f$, whenever $f \in \Hc_A$.
\begin{theorem} \label{thm:numerics}
  Let $B>0$, $V\in\Hc_B$ be real-valued, $i \in \N^\ast$ and $0 < A < B$. Let $\lambda_i$
  be the $i^{\rm th}$ lowest eigenvalue of the self-adjoint operator
  $H=-\Delta + V$ on $L^2_\per(\R,\C)$ counting multiplicities, and
  ${\mathcal E}_i={\rm Ker}(H-\lambda_i)$ the corresponding eigenspace. For
  $N$ large enough, we denote by
  $\lambda_{i,N}$ the $i^{\rm th}$ lowest eigenvalue of \eqref{eq:var_approx},
  and by $u_{i,N}$ an associated normalized eigenvector. Then, there is a
  constant $C_{i,A} \in \R_+$ such that, for $N$ large enough,
  \[
    d_{H^1_\per}(u_{i,N},{\mathcal E}_i) \lq C_{i,A} \exp\prt{-AN} \quad
    \mbox{and} \quad 0 \lq \lambda_{i,N} - \lambda_i \lq C_{i,A} \exp\prt{-2AN}.
  \]
\end{theorem}

\begin{proof}
  First, note that $-\Delta + V$ has compact resolvent so that its eigenvalues
  $\lambda_i$ are isolated.
  Let $0 < A < B$ and $A'=\frac{A+B}2$. We have
  $$
  \Forall v \in \Hc_{A'}, \quad \norm{v-\Pi_Nv}_{H^1_\per} \lq C_{A,B}
  \|v\|_{A'}\e^{-AN}  \quad $$ with $$ \quad C_{A,B}\coloneqq \left( \max_{k
      \in \Z} \frac{(1+|k|^2)\e^{2A|k|}}{w_{A'}(k)} \right)^{1/2} < \infty.
  $$
  The operator $H$ is self-adjoint on $L^2_\per(\R,\C)$ with form domain
  $H^1_\per(\R,\C)$ and we infer from Theorem~\ref{thm:lin_egval} that all the
  eigenfunctions of the operator $H$ are in $\Hc_{A'}$. Theorem~\ref{thm:numerics}
  then follows
  from classical arguments on the variational approximations of the eigenmodes
  of bounded below self-adjoint operators with compact resolvent (see
  e.g.~\cite[Theorems 8.1 and 8.2]{babuskaEigenvalueProblems1991}).
\end{proof}

\subsection{Numerical results in the linear case}\label{sec:numerics_lin}

\subsubsection{Computational framework}

In this paper, all the numerical tests are
realized with the DFTK software \cite{herbstDFTKJulianApproach2021}. This
\texttt{Julia} package uses
a planewave basis $X_N$, as defined in Section~\ref{sec:numeric}, through a
discretization parameter $\Ecut \coloneqq N^2/2$. Then, the
numerical strategy depends on the nature of the problem:
\begin{itemize}
  \item linear eigenproblems are solved with a LOBPCG solver
    (see e.g.\ \cite{nottoliRobustOpensourceImplementation2023});
  \item nonlinear eigenproblems (see Section~\ref{sec:gp})
    and nonlinear elliptic problems (see Section~\ref{sec:nonlin_ellip}) with
    source terms are solved by direct
    minimization of the corresponding energy, through the \texttt{Julia} implementation
    of the LBFGS minimization algorithm \cite{liuLimitedMemoryBFGS1989}.
\end{itemize}

\subsubsection{Results in the linear case}

In this section, we provide some numerical experiments that illustrate
Theorems~\ref{thm:lin_egval} and \ref{thm:numerics}. To this end, we consider the
potential $V$ defined by
\[
  \Forall x \in [0,2\pi],\quad
  V(x) = \frac{1}{\gamma + \sin^2(x)},
\]
with $\gamma = 1/500$.
The analytic continuation of $V$ has branching points at $\pi\Z \pm \i B_0$ with
\[
  B_0 = {\rm arcsinh}(\sqrt{\gamma})\approx0.0447,
\]
and a direct calculation shows that $V\in\Hc_B$ for any $B<B_0$. From the
previous results, we therefore expect the eigenfunctions of $-\Delta + V$ to
belong to $\Hc_{B}$ for any $B<B_0$ and the planewave approximation to converge with rate
proportional to $\exp(-2B_0N)$ for the eigenvalues and proportional to
$\exp(-B_0N)$ for the eigenfunctions in the $H^1$ norm.

\medskip

In Figure~\ref{fig:fourier_linear}, we compute numerically the ground state $u_1$ of
$-\Delta + V$ with discretization parameter
$\Ecut = 5\cdot10^5$. In order to be able to properly spot the convergence
of the Fourier coefficients, we use quadruple precision and the tolerance of the linear solver is set to
$10^{-20}$. As expected, the Fourier coefficients of (the
numerical approximation of) $u_1$
decrease with rate of order $\sqrt{w_{B_0}(k)}$, which is confirmed by looking
at the convergence of two successive nonzero Fourier coefficients, since
\[
  \lim\limits_{k\to\infty}
  \log\prt{\sqrt{\frac{\cosh\prt{2B_0k}}{\cosh\prt{2B_0(k+1)}}}} = -B_0.
\]

\medskip

In Figure~\ref{fig:discretization}, we solved the same problem for various $N$'s,
and computed the error on the eigenvalues and the eigenvectors in the $H^1$ norm
with respect to a reference solution obtained with $N=\sqrt{2\Ecut} = 1000$. This time, the
tolerance of the linear solver is set to $10^{-13}$ and we use standard double
precision. The convergence
rates predicted by Theorem~\ref{thm:numerics} are observed, as expected. We can also
remark that the error on the eigenvalues seems to decrease slightly faster (with
an additional $1/N^2$ factor), the asymptotic rate being still given by the
exponential factor $\exp(-2B_0N)$.

\begin{figure}[h!]
  \includegraphics[width=\linewidth]{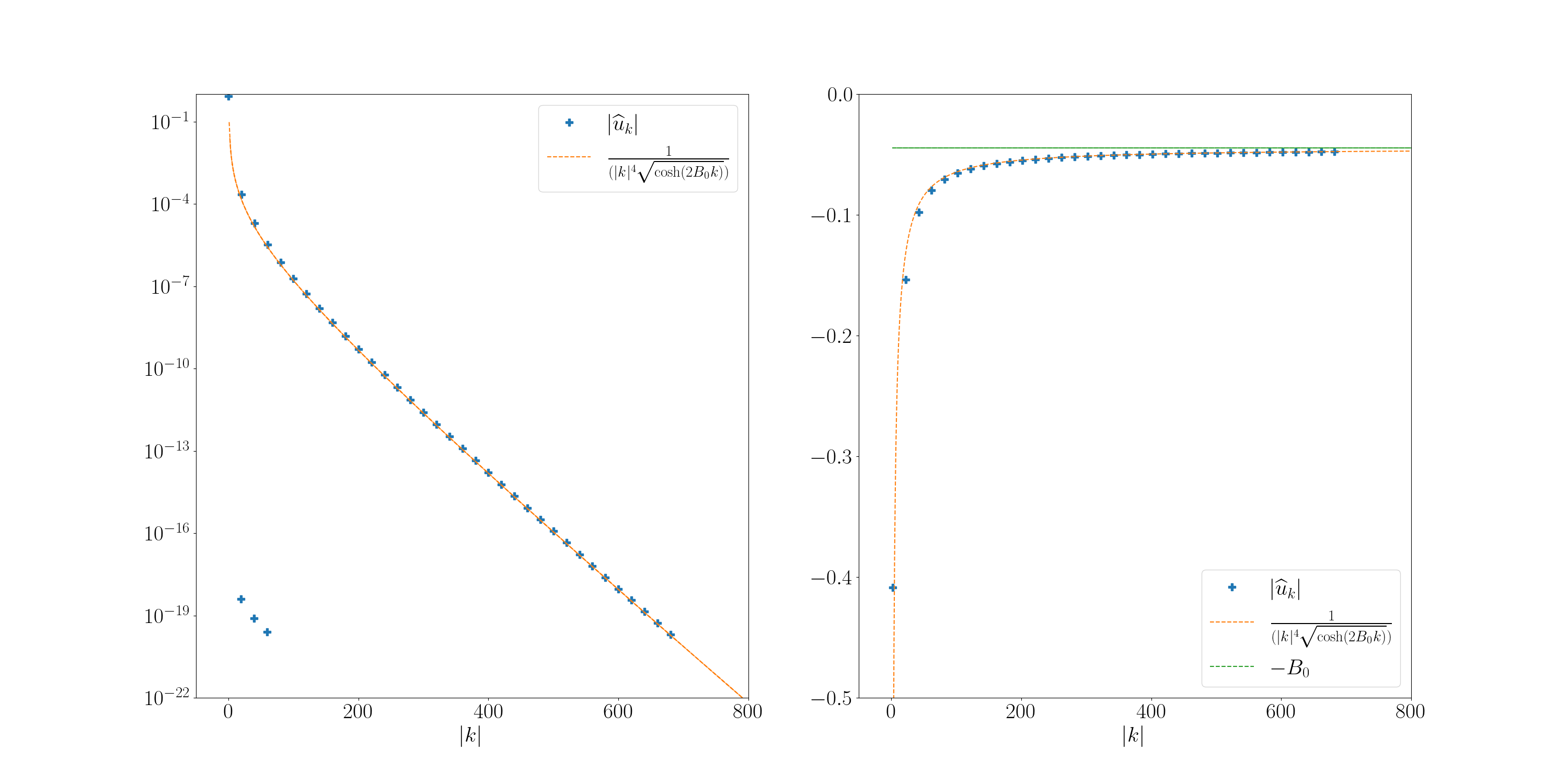}
  \caption{[Linear eigenvalue problem] (Left) Fourier coefficients of (the
    numerical approximation of) $u_1$:
    they seem to decrease like $1/(\abs{k}^\alpha\sqrt{\cosh{2B_0k}})$ with
    $\alpha = 4$, the prefactor $2B_0$ in the cosh being confirmed by the plot
    on the right. (Right) Logarithm of the ratio of two successive nonzero
    Fourier coefficients.}
  \label{fig:fourier_linear}
\end{figure}

\begin{figure}[h!]
  \includegraphics[width=\linewidth]{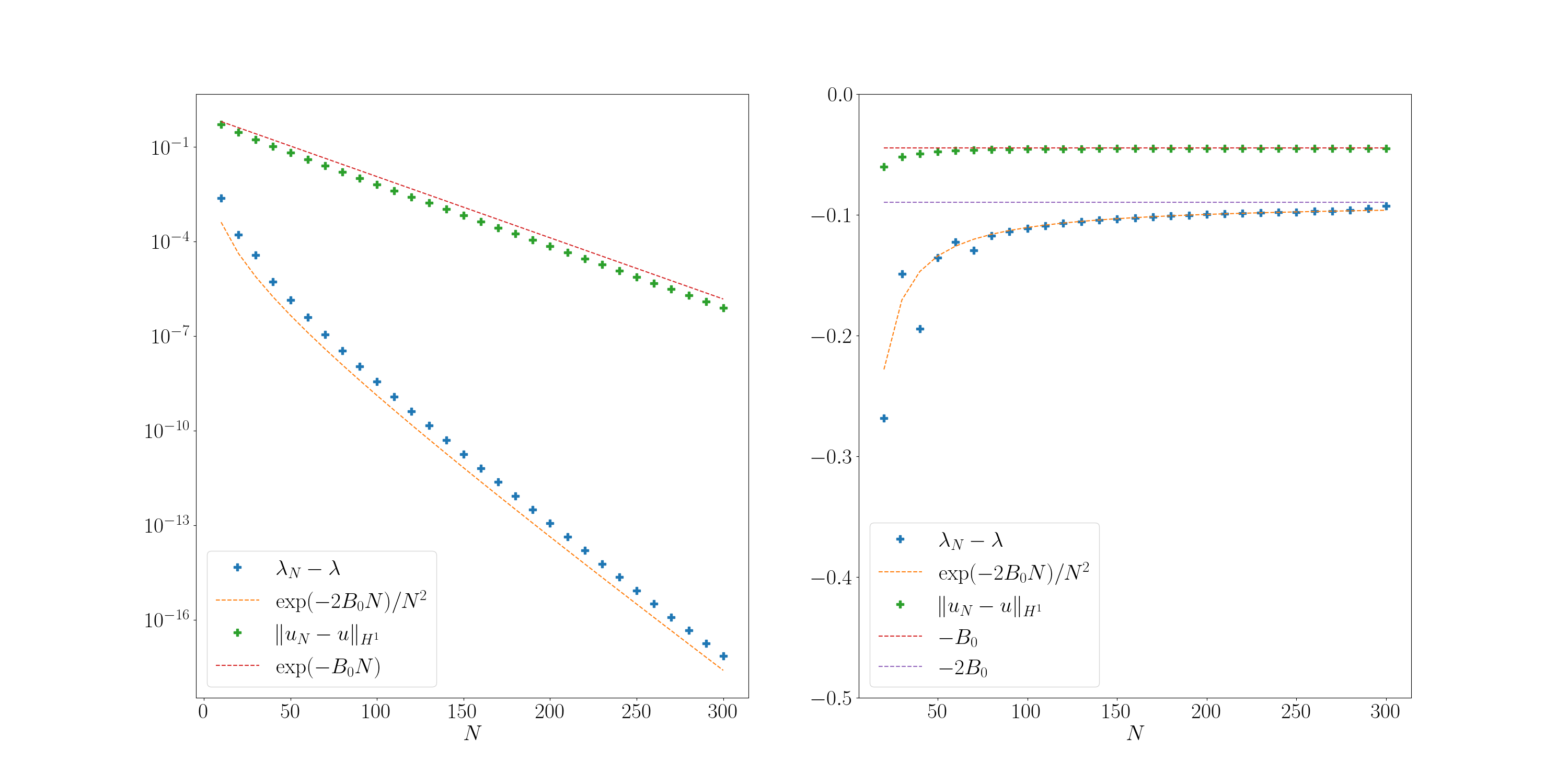}
  \caption{[Linear eigenvalue problem] (Left) Discretization error for the
    first eigenvalue and the associated eigenvector in the $H^1$ norm. Both
    errors seem to converge with rates given by Theorem~\ref{thm:numerics}, the
    prefactors in the $\cosh$ being confirmed by the plot on the right.
    (Right) Logarithm of the ratio between discretization errors for various
    $N$'s.}
  \label{fig:discretization}
\end{figure}

\newpage
\section{The nonlinear case}
\label{sec:nonlin}

In the nonlinear case, the analyticity strip of the solution can be much smaller
than the one of the data.
First, we report numerical simulations on a nonlinear periodic Gross--Pitaevskii
eigenvalue problem with a cubic nonlinearity and a  potential admitting an
entire continuation, illustrating the fact that the eigenfunction belongs to
$\Hc_A$ for some $A>0$, but does not seem to be entire.
Next, we study a nonlinear Schrödinger equation with entire potential and source
term and cubic nonlinearity.
We show that the solution is not entire and we provide an upper bound of
the width of the horizontal analyticity strip of the solution.

\subsection{Analyticity on a strip}

We consider here the nonlinear eigenvalue problem
\begin{equation}\label{eq:GP_eig_anal}
  -\Delta u + Vu + |u|^2u = \lambda u, \quad \norm{u}_{L^2_\per(\Rb,\Cb)}=1,
\end{equation}
where $V\in\Hc_B$ for any $B>0$ is real-valued. Similarly to the linear case, we
consider the planewave variational approximation:
using again $$X_N~=~\mbox{Span}(e_k,~\;~|k|~\lq~N),$$ we obtain the
finite-dimensional problem
\begin{equation}
  \label{eq:var_approx_nl}
  \begin{cases}
    \text{seek } (u_N,\lambda_N) \in X_N \times \Rb \text{ such that }
    \norm{u_N}_{L^2_\per(\Rb,\Cb)}=1 \text{ and }\\
    \displaystyle{\Forall v_N \in X_N,\quad \int_0^{2\pi}\wb{\nabla
        u_N} \cdot \nabla v_N +
      \int_0^{2\pi} V \wb{u_N} v_N
      + \int_0^{2\pi} |u_N|^2\wb{u_N}v_N = \lambda_N \int_0^{2\pi}
      \wb{u_N} v_N.}
  \end{cases}
\end{equation}
Then, the following theorem
holds, which is a simple corollary of known results on the analyticity of
elliptic partial differential equations.
Note that this result only yields the existence of a finite strip of analyticity in the complex
plane. However, even if the data is entire, there is \emph{a priori} no reason for the
solution to be entire too, as shown by the counter-examples that follow.
\begin{theorem}
  \label{thm:anal_nl}
  Let $V$ real-valued be in $\Hc_B$ for any $B>0$ and $(u,\lambda)\in
  H^2_\per(\R,\R)\times\R$
  be the ground-state of \eqref{eq:GP_eig_anal}, uniquely
  defined under the assumption that $u>0$. Then, there exists $A>0$ such that
  $u\in\Hc_A$.
  Moreover, if $(u_N,\lambda_N)$ is the variational approximation of
  $(u,\lambda)$ in $X_N\times\R$, then
  \[
    \Exist C_A>0,\quad \norm{u-u_N}_{H^1_\per} \lq C_A\exp(-AN).
  \]
\end{theorem}

\begin{proof}
  Existence and uniqueness of $u$ is a classical result, proved for instance in
  \cite[Appendix]{cancesNumericalAnalysisNonlinear2010}. In addition, as
  $V\in H^s_\per(\R,\R)$ for any $s>0$, we know from
  \cite[Theorem 2]{cancesNumericalAnalysisNonlinear2010} and by an immediate
  bootstrap argument that $u$ also belongs to $H^s_\per(\R,\R)$ for any $s>0$.
  Therefore, $u\in C^\infty_\per(\R,\R)$.

  \medskip

  Knowing that $u\in C^\infty_\per(\R,\R)$ yields that $u$ is real-analytic on a
  neighborhood of every point $x_0$ of the periodicity cell $[0,2\pi]$,
  according to well-known results on the analyticity of solutions to nonlinear
  elliptic equations, which we recall here for the reader's
  convenience (see for instance \cite[Theorem 1.1]{blattAnalyticitySolutionsNonlinear2020},
    \cite[Theorem 1]{hashimotoRemarkAnalyticitySolutions2006} or
    \cite{friedmanRegularitySolutionsNonLinear1958,morreyAnalyticitySolutionsAnalytic1958,morreyAnalyticitySolutionsAnalytic1958b}).
  \begin{theorem}
    Let $\Omega$ be an open subset of $\R^n$, $D$ an open subset of $\R^n \times
    \R \times \R^n \times \R^{n \times n}_{\rm sym}$, $F : D \ni (x,v,g,M)
    \mapsto F(x,v,g,M) \in \R$ a real analytic function, and $u \in
    C^\infty(\Omega)$ a real-valued function such that for all $x \in \Omega$,
    $(x,u(x),\nabla u(x), \nabla^2 u(x)) \in D$. Assume that $u$ solves
    $F(x,u(x),\nabla u(x), \nabla^2 u(x))=0$ and that this equation is elliptic
    in the sense that
    \[
      \forall (x,\xi) \in \Omega \times \R^n, \; \xi \neq 0, \quad \xi^T \nabla_MF(x,u(x),\nabla u(x),\nabla^2u(x)) \xi \neq 0.
    \]
    Then $u$ is real-analytic in $\Omega$.
  \end{theorem}
  By a compactness argument, we therefore have that $u$ is analytic on a strip of
  size $A>0$ around the real axis, for some $A>0$.

  \medskip

  We now consider the variational approximation $(u_N,\lambda_N)$ of
  $(u,\lambda)$ in $X_N\times \R$. We know from \cite[Theorem
  1]{cancesNumericalAnalysisNonlinear2010} that there exists $C>0$ such that
  \[
    \norm{u-u_N}_{H^1_\per} \lq C \norm{u-\Pi_Nu}_{H^1_\per}.
  \]
  The result then follows similarly to the proof of Theorem~\ref{thm:numerics}.
\end{proof}

\begin{remark}
  In order to establish the convergence rate of the eigenvalues of the ground-state of
  \eqref{eq:GP_eig_anal}, one can follow the proof of \cite[Theorem
  2]{cancesNumericalAnalysisNonlinear2010}.
  Using in particular equations (50) and (53) from this reference, with the negative
  Sobolev norms replaced by the dual norms of $\Hc_A$, one can obtain that there
  exists $C_A>0$ such that $ \abs{\lambda-\lambda_N} \lq C_A\exp(-2AN) $.
\end{remark}

\subsection{Counter-examples}

We analyze in this section two counter-examples of Theorems \ref{thm:lin_bounds} and
\ref{thm:lin_egval} based on the nonlinear
Gross--Pitaevskii equation with entire potentials and source terms, for which the
solutions are not entire. We also present numerical tests to support our
analysis.

\subsubsection{A nonlinear Gross--Pitaevskii eigenvalue problem}
\label{sec:gp}

We study in this section the following nonlinear eigenvalue problem:
\begin{equation}
  \label{eq:GP_eig}
  \begin{cases}
    - \varepsilon \Delta u_\varepsilon + \gamma\cos(x) u_\varepsilon +
    |u_\varepsilon|^2u_\varepsilon = \lambda_\varepsilon u_\varepsilon
    \quad\mbox{in } H^1_\per(\Rb,\Cb), \\
    \norm{u_\varepsilon}_{L^2_{\rm per}} = 1, \quad u_\varepsilon > 0,
  \end{cases}
\end{equation}
with entire data. In view of the results we obtained in the linear
case, we could expect the solution $u_\varepsilon$ to
\eqref{eq:GP_eig} (which we know to exist) to be entire. However, our
numerical experiments suggest on the contrary that $u_\varepsilon$ is
analytic only on a band of finite size.

\medskip

The potential $V$ being real-valued, $u_\varepsilon$ is actually real-valued too
and the singular limit $\varepsilon\to0$ yields the algebraic equation $\gamma
\cos(x)u_0(x) + u_0(x)^3 = \lambda_0 u_0(x)$. If we assume that $u_0$ does not
vanish on the unit cell $[0,2\pi]$, we can divide by $u_0(x)$ and then integrate
over $[0,2\pi]$: using the normalization condition, we obtain $\lambda_0 =
1/2\pi$. This yields
\begin{equation}\label{eq:eg_u0}
  u_0(x) = \sqrt{\frac1{2\pi} - \gamma\cos(x)},
\end{equation}
which is indeed bounded away from zero as soon as $\beta_0 \coloneqq 1 / (2\pi
\gamma) > 1$, which we take in the following. The analytic continuation of $u_0$, still denoted by $u_0$, satisfies
\[
  \gamma \cos(z) u_0(z) + u_0(z)^3 = \frac1{2\pi} u_0(z),
\]
with $\sqrt{\cdot}$ in \eqref{eq:eg_u0} being the continuation of the square root with
branch cut $\R_-$. The maximal horizontal strip of the complex plane on which
$u_0$ is analytic is $\R + \i(-B_0,B_0)$, where $B_0$ is such that $\cos(\pm
\i B_0) = \beta_0$, that is to say $B_0 = \log\prt{\beta_0 +
  \sqrt{\beta_0^2-1}}$. More precisely, the function $z\mapsto u_0(z)$ admits
two branching points at $z_{\pm} = \pm\i B_0$, for which $u_0(z_\pm) = 0$ and
$$
\left| \frac{\d u_0}{\d t}(t z_\pm)\right| =
\left|\frac{\gamma\sin(tz_\pm)}{2\sqrt{\frac1{2\pi}- \gamma\cos(tz_\pm)}} \right|
\to \infty \quad \mbox{when $\R \ni t \to 1_-$},
$$
see Figure~\ref{fig:u0_eg}.

\begin{figure}[h!]
  \includegraphics[width=\linewidth]{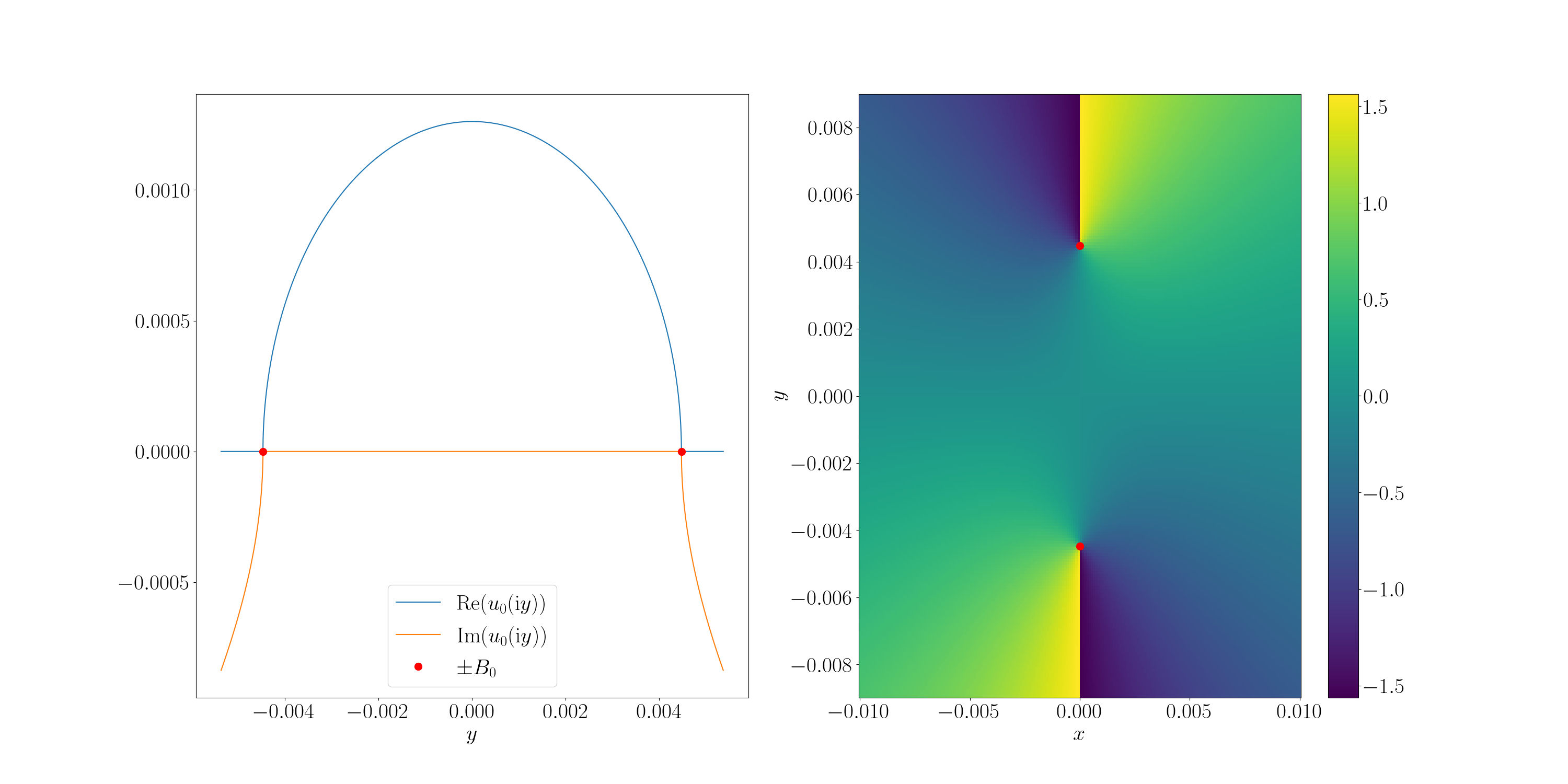}
  \caption{[Nonlinear eigenvalue problem] Analytic continuation of $u_0$ for $\beta_0 = 1.00001$, for which
    $B_0 \approx 0.0045$.
    (Left) Imaginary and real parts of
    $y \mapsto u_0(\i y)$.  A discontinuity of the derivative appears at the expected positions.
    (Right) Phase of $z \mapsto u_0(z)$ where $z=x+\i y$.  Branching points, in red,
    appear at the expected positions.}
  \label{fig:u0_eg}
\end{figure}

\medskip

When $\varepsilon > 0$, we can approximate numerically the solution of
\eqref{eq:GP_eig} using the planewave approximation presented before, with
discretization parameter $\Ecut = 10^6$ and solver tolerance $10^{-13}$.
The Fourier coefficients of the numerical
solution for various $\varepsilon$ are displayed in Figure~\ref{fig:fourier_u_eg}.
Although it is not possible to rule out a transition to a different asymptotic
regime in the limit of extremely small values of $\varepsilon$ or very large values
of $k$ because of finite numerical accuracy,
our numerical results are compatible with the decay of the Fourier
coefficients of $u_\varepsilon$ as $1/(|k|^{3/2}\sqrt{\cosh(2B_0k)})$.
This leads to the conclusion that the solution $u_\varepsilon$ is
analytic on a band strip of size close to $B_0$,
which is in contradiction with the results from the linear case. In
the next section, we study a similar nonlinear elliptic equation where we
reach similar conclusions and where, in addition, we can provide an upper bound
of the width of the horizontal strip on which the solutions are analytic.

\begin{figure}[h!]
  \centering
  \includegraphics[width=\linewidth]{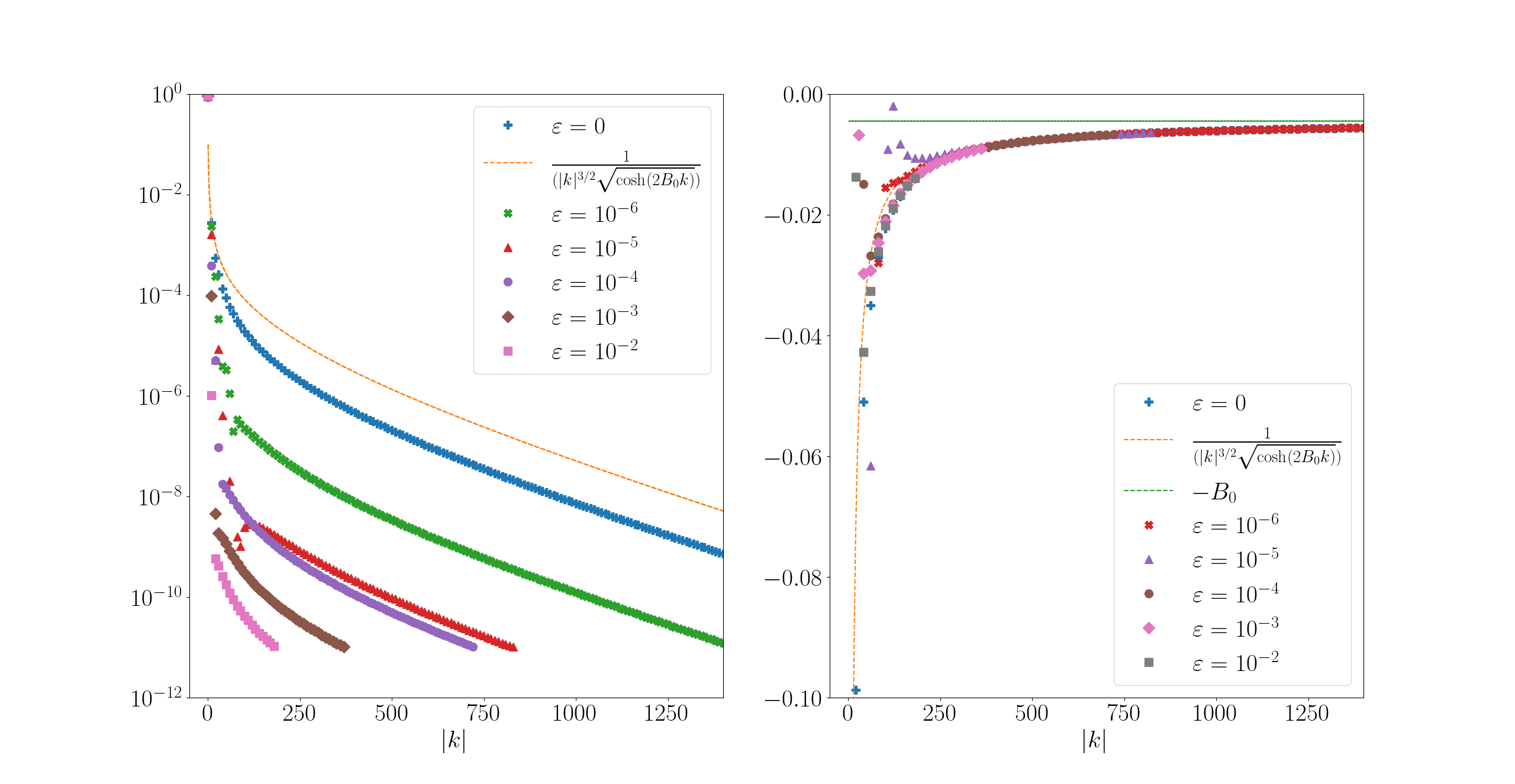}
  \caption{[Nonlinear eigenvalue problem] (Left) Fourier coefficients of (the
    numerical approximation of) $u_\varepsilon$: they seem to decrease
    like $1/(\abs{k}^{\alpha}\sqrt{\cosh(2B_0k)})$ with $\alpha=3/2$, the
    prefactor $2B_0$ in the $\cosh$ being confirmed by the plot on the right. (Right) Logarithm
    of the ratio of two successive nonzero Fourier coefficients.}
  \label{fig:fourier_u_eg}
\end{figure}

\subsubsection{A nonlinear elliptic equation}\label{sec:nonlin_ellip}

Still in the perspective of studying nonlinear elliptic problems with analytic
data, we now consider the nonlinear periodic elliptic equation with cubic
nonlinearity
\begin{equation}
  \label{eq:GP}
  -\varepsilon\Delta u_\varepsilon + u_\varepsilon +
  |u_\varepsilon|^2u_\varepsilon = f \quad
  \mbox{in } H^1_\per(\R,\C),
\end{equation}
where $\varepsilon > 0$, and $f:\R \to \R$ is a real-analytic $2\pi$-periodic
function admitting an entire continuation, still denoted by $f$, to the
complex plane.
We will show that, in this particular case, the same kind of results we showed
for the linear case are not true any more
and we provide an estimation of the width of the horizontal analyticity strip of the
solution, which is finite even though the source term $f$ is entire.

\medskip

The right-hand-side $f$ being real-valued, $u_\varepsilon$ is also real-valued
and the singular limit $\varepsilon = 0$ gives rise to the algebraic equation
$u_0(x) + u_0(x)^3 = f(x)$, which has a unique real solution for each
$x \in \R$. The latter can be computed by Cardano's formula: the discriminant
of the cubic equation is
\[
  R(x) = -\prt{4 + 27f^2(x)} < 0,
\]
so that
\begin{equation}
  \label{eq:cardan}
  u_0(x) = \sqrt[3]{\frac{1}{2}\prt{f(x) + \sqrt{\frac{-R(x)}{27}}}} +
  \sqrt[3]{\frac{1}{2}\prt{f(x) - \sqrt{\frac{-R(x)}{27}}}},
\end{equation}
the other two roots being complex conjugates with nonzero imaginary
parts.
The function $u_{0}$ can be analytically continued from the real axis
upwards in the complex plane as long as $R(z)$ does not touch zero; it
has branching points (with an exploding first derivative) at the
points where $R(z) = 0$. In the rest of this section, we will use the function
$f(x) = \mu \sin(x)$, for a fixed $\mu > 0$. The analytic continuation
of $u_{0}$ has branching points at $2\pi \Z \pm \i B_0$, with
\begin{align*}
  B_{0} = {\rm arcsinh} {\sqrt{\frac 4 {27 \mu^{2}}}}> 0.
\end{align*}
In particular, although the source term $f$ has an entire continuation, the
solution $u_0$ given by \eqref{eq:cardan} does not.

\medskip

When $\varepsilon > 0$, we can approximate numerically the solution to
\eqref{eq:GP} with the planewave approximation introduced
before, with discretization parameter $\Ecut = 10^6$ and solver tolerance
$10^{-13}$. The Fourier coefficients of the numerical
solution are then displayed in Figure~\ref{fig:fourier_u} for various
$\varepsilon$: the plots suggest that increasing $\varepsilon$
increases the width $B_\varepsilon$ of the horizontal analyticity strip of
$u_\varepsilon(z)$, but does not make it entire.

\begin{figure}[h!]
  \includegraphics[width=\linewidth]{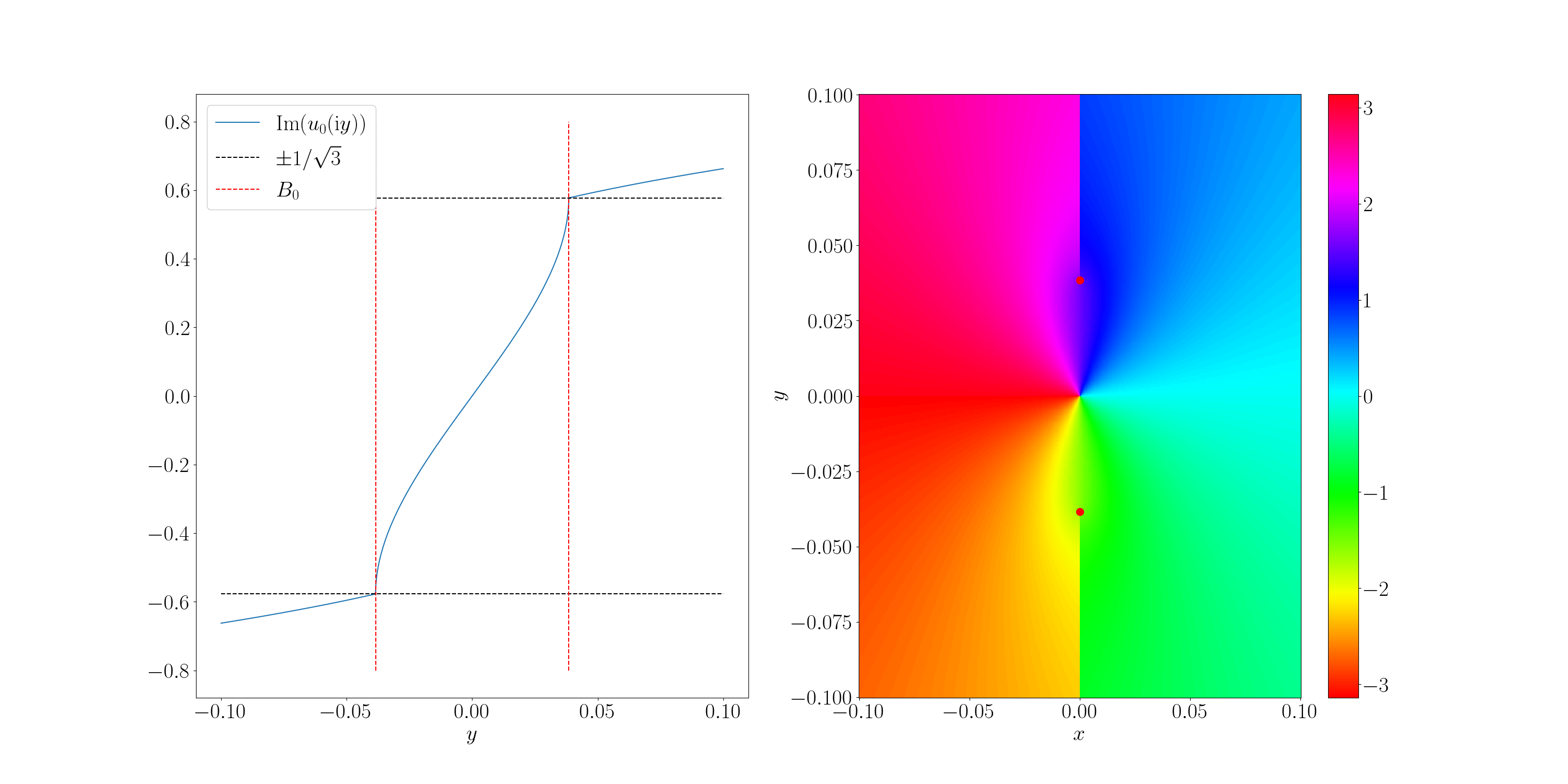}
  \caption{[Nonlinear elliptic equation] Analytic continuation of $u_0$ for $\mu=10$, for which
    $B_0 \approx0.0385$.
    (Left) Imaginary part of
    $y \mapsto u_0(\i y)$.  Discontinuities also appear at the expected
    positions $\pm B_0$. (Right) Phase of $z \mapsto u_0(z)$ where $z=x+\i y$.
    Branching points, in red, appear at the expected positions $\pm \i B_0$.}
  \label{fig:cardan_eps0}
\end{figure}

\begin{figure}[h!]
  \includegraphics[width=\linewidth]{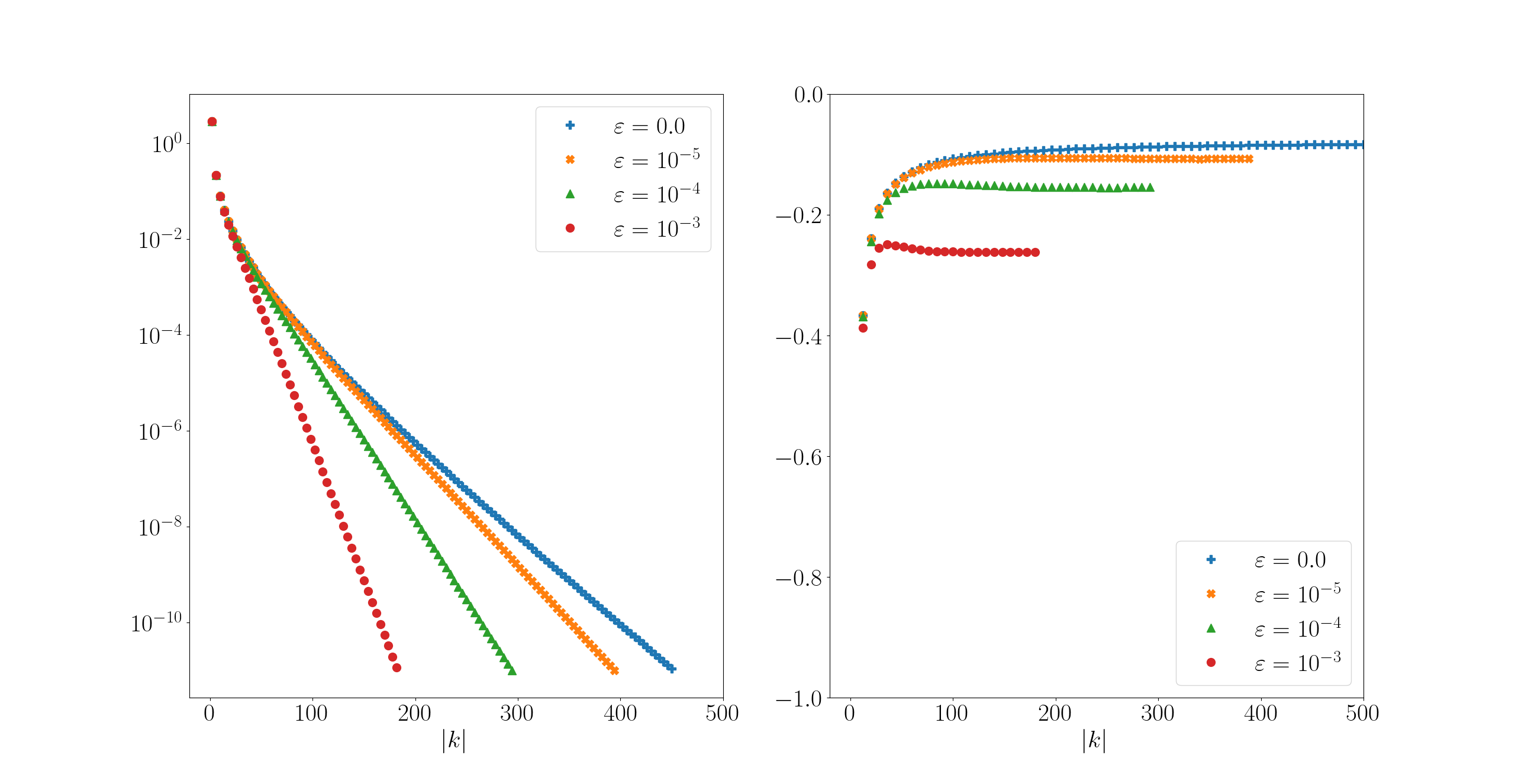}
  \caption{[Nonlinear elliptic equation] (Left) Fourier coefficients of (the
    numerical approximation of) $u_\varepsilon$. (Right) Logarithm
    of the ratio of two successive nonzero Fourier coefficients.}
  \label{fig:fourier_u}
\end{figure}

\medskip

In this particular case, we are able to obtain an upper bound of the value of
$B_\varepsilon$ by an ODE technique. Let
$\varphi_\varepsilon(y) \coloneqq u_\varepsilon(\i y)$ be the value of the
solution along the imaginary axis. Using the
Cauchy--Riemann equations, we see that $\varphi_{\varepsilon}$ satisfies the
following second-order ODE:
\[
  \begin{cases}
    \varepsilon {\varphi}_\varepsilon''(y) + \varphi_\varepsilon(y) +
    \varphi_\varepsilon^3(y) =
    \i \mu\sinh(y),\\
    \varphi_\varepsilon(0) = u_\varepsilon(0) = 0,\quad
    \varphi'_\varepsilon(0) = \i u'_\varepsilon(0),
  \end{cases}
\]
with $u'_\varepsilon(0) \in \R$.
Decomposing $\varphi_\varepsilon$ in its real part $\theta_\varepsilon$ and
imaginary part $\psi_\varepsilon$, we see that $\theta_\varepsilon$ satisfies the
equation
\begin{align*}
  \begin{cases}
    \varepsilon {\theta}''_\varepsilon + \theta_\varepsilon +
    \theta_\varepsilon^3
    -3 \theta_\varepsilon\psi_\varepsilon^2 = 0, \\
    \theta_\varepsilon(0) = 0,\quad \theta'_\varepsilon(0) = 0,
  \end{cases}
\end{align*}
and therefore vanishes. As a consequence, $u_{\varepsilon}$ is purely imaginary along the
imaginary axis (as could have been anticipated by the symmetries
$u_{\varepsilon}(-z)~=~-~u(z)$,
$u_{\varepsilon}(\overline{z}) = \overline{u_{\varepsilon}(z)}$).
We are then left with studying the imaginary part $\psi_{\varepsilon}$,
which satisfies the ODE
\begin{align}
  \begin{cases}
    \varepsilon {\psi}''_\varepsilon + \psi_\varepsilon - \psi_\varepsilon^3 = \mu\sinh,\\
    \psi_\varepsilon(0) = 0,\quad
    \psi'_\varepsilon(0) = u'_\varepsilon(0).
  \end{cases}
  \label{eq:psi}
\end{align}
If we can prove that $\psi_\varepsilon$ becomes non-analytic at a finite
$0 < Y_\varepsilon < \infty$, this will imply that the width
$B_\varepsilon$ of the horizontal analyticity strip of $u_\varepsilon$
satisfies $B_\varepsilon \lq  Y_\varepsilon < \infty$ and is therefore finite
even though the data in \eqref{eq:GP} is entire.
To this end, we prove in the Appendix B the following lemma, which gives a
sufficient condition for the non-analyticity of $\psi_\varepsilon$ in finite
time.
\begin{lemma}\label{lem:comp}
  Let $Y_\varepsilon\in\R_+\cup\{\infty\}$ be the maximum time of definition of $\psi_\varepsilon$.
  If, for some $\eta>0$ small enough, there exists $y_\eta\gq B_0$ such that
  $\psi_\varepsilon(y_\eta) = 1+\eta > 1$ and
  $\psi_\varepsilon'(y_\eta) > 0$, then $Y_\varepsilon < \infty$ and
  $\psi_\varepsilon(y)\to\infty$ when $y\to Y_\varepsilon$.
\end{lemma}

To prove that the sufficient condition from Lemma~\ref{lem:comp} is satisfied,
we introduce the set
\begin{equation*}
  X_\mu = \set{
    (y,\upsilon) \in \Rb^2,\ \mu\sinh(y) - \upsilon + \upsilon^3 \gq 0
  }.
\end{equation*}
This set is such that, if $(y,\psi_\varepsilon(y))$ lies strictly inside
$X_\mu$, then
$\psi_\varepsilon$ is locally strictly convex. For $y < B_0$,
$\psi_\varepsilon$ might oscillate on both sides of the boundary of $X_\mu$
\cite{vrabelDuffingTypeOscillatorBounded2013}. We
investigated numerically the behavior of this function for the set of
parameters used in Figure~\ref{fig:Xmu} ($\varepsilon = 0.1$, $\mu = 0.5$), and
observed that $0 < \psi_\varepsilon(B_0) < \frac 1{\sqrt 3}$ and
$\psi'_\varepsilon(B_0) > 0$.  This numerical observation can be trusted as
the ODE satisfied by $\psi'_\varepsilon$ on the interval $[0,B_0]$ for
$\varepsilon = 0.1$ and $\mu = 0.5$ is not stiff. It is therefore easy to
solve it numerically with high accuracy with {\it a posteriori} error
estimates guaranteeing that $\psi_\varepsilon(B_0)$ is indeed strictly between
$0$ and $\frac 1{\sqrt 3}$, and  $\psi'_\varepsilon(B_0)$ is positive.
Therefore, $\psi_\varepsilon$ is locally strictly convex at $y=B_0$. Given the shape of
$X_\mu$ (see Figure~\ref{fig:Xmu} in appendix), $\psi_\varepsilon$ lies, for any
$B_0<y<Y_\varepsilon$,
above its tangent at $y=B_0$, whose slope is $\psi'_\varepsilon(B_0)>0$.
Strict convexity allows one to conclude that
for any $\eta>0$, there exists $y_\eta > B_0$ such that
\[
  \begin{cases}
    \psi_\varepsilon(y_\eta)=1+\eta>1, \\
    \psi_\varepsilon'(y_\eta) > \psi_\varepsilon'(B_0) > 0.
  \end{cases}
\]
Any $\eta>0$ is therefore suitable to apply
Lemma~\ref{lem:comp} and conclude that $\psi_\varepsilon$ blows up in finite time
$Y_\varepsilon$. We can deduce from these investigations that $u_\varepsilon$
is analytic only on a horizontal strip of finite width of the complex plane although the source term
$f$ is an entire function: our results in the linear case are therefore no
longer valid in general in the nonlinear case.

\section*{Conflict of interest}
All the authors declare that they have no conflicts of interest.

\section*{Data availability}
All the scripts used to generate the plots of Figures \ref{fig:fourier_linear},
\ref{fig:discretization}, \ref{fig:u0_eg},
\ref{fig:fourier_u_eg}, \ref{fig:cardan_eps0} and \ref{fig:fourier_u}  are
available online at
\begin{center}
  \url{https://github.com/gkemlin/analytic\_potentials}.
\end{center}

\section*{Acknowledgements}
The authors would like to thank Geneviève Dusson and Michael F. Herbst for
fruitful discussions. This project has received funding
from the European Research Council (ERC) under the European Union’s Horizon
2020 research and innovation programme (grant agreement No 810367).
The authors would also like to thank the anonymous reviewers for their comments
and suggestions.

\appendix

\section{Extension to the multidimensional case with application to Kohn--Sham
  models.}
\label{sec:multi_d}

The goal of this section is to extend the previous results to the
multidimensional case and apply them to the linear version of the Kohn--Sham
equations \eqref{eq:KS-DFT}.
To this end, consider a Bravais lattice $\LL=\Z \bm a_1+ \cdots + \Z \bm a_d$ where
$\bm a_1,\dots, \bm a_d$ are linearly independent vectors of $\R^d$
($d=3$ for KS-DFT). Up to an affine change of variables (which preserves
analyticity), we can take without loss of generality $\LL = 2\pi \Z^{d}$.
We denote
by $\Omega=[0,2\pi)^{d}$ a unit cell, by $\LL^* = \Z^{d}$ the
reciprocal lattice, by $e_{\bm G}(\bm x)=|\Omega|^{-1/2}\e^{\i\bm G \cdot \bm x}$ the
Fourier mode  with wavevector $\bm G \in \LL^*$, and by
$$
H^s_{\per,\LL} \coloneqq \set{ u= \sum_{\bm G \in \LL^*} \widehat u_{\bm G}
  e_{\bm G}
  \; \middle| \; \sum_{\bm G \in \LL^*} (1+|\bm
  G|^2)^s |\widehat u_{\bm G}|^2 < \infty }
$$
the $\LL$-periodic Sobolev spaces endowed with their usual inner products.
All the arguments in Sections~\ref{sec:lin_pb}-\ref{sec:numeric} can be
extended to the multidimensional case by introducing the Hilbert spaces
\[
  \Hc_{A,\LL} \coloneqq \set{u \in L^2_{\per,\LL} \;\middle|\;
    \sum_{\bm{G}\in \LL}  {w}_{A,\LL} (\bm{G}) \abs{\wh{u}_{\bm{G}}}^2 <
    \infty},\quad
  (u,v)_{A,\LL} \coloneqq \sum_{\bm{G}\in\LL^*}  {w}_{A,\LL}(\bm{G})
  \overline{\wh{u}_{\bm{G}}}  \wh{v}_{\bm{G}},
\]
where
$\displaystyle{{w}_{A,\LL}(\bm{G}) = \sum_{n=1}^d w_A(G_{n})}$.
Each $u \in \Hc_{A,\LL}$ can be extended to an analytic
function $u(z_1,\dots,z_d)$ of $d$ complex variables defined on a
neighborhood on $\R^d$, and it holds
\[
  \sum_{\bm{G}\in\LL^*}  w_{A,\LL}(\bm{G})  \abs{\wh{u}_{\bm{G}}}^2 =
  \frac{1}{2} \sum_{j=1}^d \int_{\Omega} \abs{{u}(\bm{x}+\i e_{j})}^2 +
  \abs{{u}(\bm{x}-\i  e_{j})}^2 \d \bm{x}.
\]
with $e_{1},\dots, e_d$ the canonical basis vectors.
The approximation space $X_{N,\LL}$ is then defined as
$$
X_{N,\LL} \coloneqq \mbox{Span} (e_{\bm G}, \; \bm G \in \LL^*, \; |{\bm G}| \lq N ),
$$
and the inverse $T_{22,\LL}^{-1}$ of the restriction $T_{22,\LL}$ of the
operator $-\Delta$ on $L^2_{\per,\LL}$ to the invariant subspace
$X_{N,\LL}^\perp = \mbox{Span} (e_{\bm G}, \; \bm G \in \LL^*, \; |{\bm G}| > N )$ satisfies
$$
\|T^{-1}_{22,\LL} \|_{\Lc(X_{N,\LL}^\perp)} = \|T^{-1}_{22,\LL} \|_{\Lc(\Hc_{A,\LL}
  \cap X_{N,\LL}^\perp)} \lq N^{-2}.
$$
The proofs of Theorem~\ref{thm:lin_bounds}, Theorem~\ref{thm:lin_egval} and
Theorem~\ref{thm:numerics} can thus be straightforwardly adapted to the
multidimensional case.

\medskip

Lastly, if $V \in \Hc_{B,\LL}$ for some $B > 0$, the Schr\"odinger operator
$H = -\Delta + V$ {\em considered this time as a Schr\"odinger operator on
  $L^2(\R^d,\C)$ with an $\LL$-periodic potential}, can be decomposed by the
Bloch transform \cite[Section XIII.16]{reedAnalysisOperators1978}
and its Bloch fibers are the self-adjoint operators on
$L^2_{\per,\LL}$ with domain $H^2_{\per,\LL}$ and form domain $H^1_{\per,\LL}$
defined as
$H_{\bm k} = (-\i\nabla + \bm k)^2 + V$. The following result is concerned with
the Bloch eigenmodes of the $H_{{\bm k}}$'s.
\begin{theorem} Let $B > 0$ and $V \in \Hc_{B,\LL}$. For each
  $\bm k \in \R^d$, the eigenfunctions of the Bloch fibers
  $H_{\bm k} = (-\i\nabla +\bm k)^2 + V$ of the periodic Schr\"odinger operator
  $H=-\Delta + V$ are in $\Hc_{A,\LL}$ for any $0<A<B$. Let
  $\lambda_{1,\bm k} \lq \lambda_{2,\bm k} \lq \cdots$ be the eigenvalues of
  $H_{\bm k}$ counted with multiplicities and ranked in non-decreasing order,
  and
  $\lambda_{1,\bm k,N} \lq \lambda_{2,\bm k,N} \lq \cdots  \lq
  \lambda_{d_{\LL,N},\bm k,N}$ the eigenvalues of the variational approximation
  of $H_{\bm k}$ in the $d_{\LL,N}$-dimensional space
  $$
  X_{\LL,\bm k, N}\coloneqq\mbox{\rm Span}(e_{\bm G}, \; \bm G \in \LL^*, \;
  |\bm G + \bm k| \lq N).
  $$
  Then, for each $0 < A < B$ and $i \in \N^*$, there exists a constant
  $C_{i,A} \in \R_+$ such that
  \begin{equation}\label{eq:bound_BZ}
    0 \lq \max_{\bm k \in \Omega^*} (\lambda_{i,\bm k,N} - \lambda_{i,\bm k})
    \lq C_{i,A}\e^{-2AN},
  \end{equation}
  where $\Omega^*$ is the first Brillouin zone (i.e. the Voronoi cell of the
  lattice $\LL$ of $\R^d$ containing the origin).
\end{theorem}

\begin{proof} It suffices to replace in the proofs of Theorem~\ref{thm:lin_egval}
  and Theorem~\ref{thm:numerics} $X_N$ with $X_{\LL,\bm k, N}$
  and $T_{22}$ with the restriction $T_{22,\LL,\bm k}$ of the
  operator $(-\i\nabla + \bm k)^2$ to the invariant space
  $X_{\LL,\bm k, N}^\perp$. The latter is invertible and such that
  $\| T_{22,\LL,\bm k}^{-1} \|_{\Lc(X_{\LL,\bm k, N}^\perp)} \lq N^{-2}$ and
  $\| T_{22,\LL,\bm k}^{-1} \|_{\Lc(\Hc_{A,\LL} \cap X_{\LL,\bm k, N}^\perp)} \lq N^{-2}$.
  The result then follows similarly.
\end{proof}

\begin{remark}
  As a conclusion, let us mention that, under appropriate assumptions, we can
  extend these results to the KS-DFT equations \eqref{eq:KS-DFT} with GTH
  pseudopotentials in the case where an analytic parametrization of the
  exchange-correlation
  functional is used. This is relevant for instance when studying
  condensed-phase systems, for which the periodic setting is well suited, and
  where the commonly used approximations of $V_{{\rm xc},\rho}$ are analytic on
  the positive density values that are of interest in such cases
  \cite{perdewAccurateSimpleAnalytic1992}.
  Then we can rewrite \eqref{eq:KS-DFT} as a system of elliptic PDEs:
  \[
    \begin{cases}
      \displaystyle\prt{- \frac 1 2 \Delta + V_{\rm nl} + V_{\rm loc} + V_{{\rm H},\rho}
        + V_{{\rm xc},\rho}} \varphi_i - \lambda_i \varphi_i = 0, \quad
      (\varphi_i,\varphi_{j})_{L^2_{\per,\LL}} = \delta_{ij}, \\
      \displaystyle\rho(\bm x) = \sum_{i=1}^{N_{\rm p}} |\varphi_i(\bm x)|^2, \\
      \displaystyle- \Delta V_{{\rm H},\rho}(\bm x) = 4\pi \left( \rho(\bm x) -
        \frac{1}{|\Omega|}\int_\Omega \rho \right), \quad \int_\Omega V_{{\rm
          H},\rho} = 0,
    \end{cases}
  \]
  We can then use known results for elliptic systems of PDEs
  \cite{morreyAnalyticitySolutionsAnalytic1958,morreyAnalyticitySolutionsAnalytic1958b}
  to conclude, in a way similar to what we proved in Theorem~\ref{thm:anal_nl},
  that the orbitals $\varphi_i$ belong to $\Hc_{A,\LL}$ for some $A>0$. In particular, this
  leads to the exponential convergence of planewave approximations, justifying
  the use of GTH pseudopotentials.
\end{remark}

\section{Proof of Lemma~\ref{lem:comp}}

We start by rewriting \eqref{eq:psi} as a first-order ODE on
$\Psi_\varepsilon(y) \coloneqq
\begin{bmatrix}
  \psi_\varepsilon(y) \\ \psi'_\varepsilon(y)
\end{bmatrix} \in \R^2$, starting at $y_\eta\gq B_0$:
\begin{equation}
  \label{eq:ODE}
  \Psi'_\varepsilon(y) =
  \begin{bmatrix}
    \Psi_{\varepsilon,2}(y) \\ \varepsilon^{-1} \prt{\mu\sinh(y) -
      \Psi_{\varepsilon,1}(y) + \Psi_{\varepsilon,1}^3(y)}
  \end{bmatrix}
  ,\qquad \Psi_\varepsilon(y_\eta) \coloneqq
  \begin{bmatrix}
    \psi_\varepsilon(y_\eta) \\ \psi'_\varepsilon(y_\eta)
  \end{bmatrix}=
  \begin{bmatrix}
    1+\eta \\ \psi'_\varepsilon(y_\eta)
  \end{bmatrix}.
\end{equation}
We then use the following simplified version of more general comparison results on
systems of differential inequalities
\cite{walterOrdinaryDifferentialEquations1998,wazewskiSystemesEquationsInegualites1950}.
In the sequel, the inequality $a \gq b$ for two vectors $a,b \in \R^d$ means
that $a_i \gq b_i$ for all $1 \lq i \lq d$.
\begin{theorem} (See e.g. \cite[p. 112]{walterOrdinaryDifferentialEquations1998}).
  \label{thm:comparison} Let $d\gq1$ and $G : \Rb^d \to \Rb^d$ be
  locally Lipschitz and quasimonotone in the sense that for all $X,Z \in \R^d$,
  $$
  \left(Z_i=X_i \mbox{ and } X_{j} \lq Y_{j} \mbox{ for $j\ne i$} \right) \quad
  \Rightarrow \quad \left( G(X) \lq G(Z)\right).
  $$
  Let $0 \lq y_0 < y_{\rm M} \lq \infty$ and
  $\Phi \in C^1 ([y_0,y_{\rm M}),\R^d)$ and
  $\Psi \in C^1 ([y_0,y_{\rm M}),\R^d)$ satisfying respectively the ODE
  \[
    \Phi'(y) = G(\Phi(y)),\qquad \Phi(y_0) \in\Rb^d,
  \]
  and the differential inequality
  \[
    \Psi'(y) \gq G(\Psi(y)),\qquad \Psi(y_0) = \Phi(y_0).
  \]
  Then we have
  \[
    \Psi(y) \gq \Phi(y).
  \]
  on $[y_0,y_{\rm M})$.
\end{theorem}

To apply this result to \eqref{eq:ODE}, we introduce the function
$G_\varepsilon : \R^2 \to \R^2$ defined by
$$
\Forall X = \begin{bmatrix}
  X_1 \\ X_2
\end{bmatrix}  \in \R^2, \quad G_\varepsilon(X) = \begin{bmatrix}
  X_2 \\ \varepsilon^{-1}(-X_1+X_1^3)
\end{bmatrix},
$$
and
the maximal solution $\Phi_\varepsilon$  to
\begin{equation}\label{eq:comp}
  \Phi'_\varepsilon(y) =
  G_\varepsilon( \Phi_\varepsilon(y)), \qquad
  \Phi_\varepsilon(y_\eta) =
  \begin{bmatrix}
    1+\eta \\ \psi_\varepsilon'(y_\eta)
  \end{bmatrix}.
\end{equation}
As $\sinh(y) \gq 0$ for all $y \gq 0$, we have the following differential
inequality:
\[
  \Psi'_\varepsilon(y) \gq G_\varepsilon(\Psi_\varepsilon(y)),
  \qquad \Psi_\varepsilon(y_\eta) = \Phi_\varepsilon(y_\eta).
\]
Note that, by convexity, $\psi_\varepsilon(y) \gq \psi_\varepsilon(y_\eta) > 1$
for $y\gq y_\eta$:
$G_\varepsilon$ is indeed quasimonotone on the domain of interest. We are now
ready to compute an upper bound of $Y_\varepsilon$ with the use of
Theorem~\ref{thm:comparison}, which yields
\[
  \Forall y \gq y_\eta,\ \Psi_\varepsilon(y) \gq \Phi_\varepsilon(y).
\]
This leads us to the study of the ODE
\[
  \begin{cases}
    \varepsilon \phi''_\varepsilon = -\phi_\varepsilon + \phi_\varepsilon^3 \\
    \phi_\varepsilon(y_\eta) = 1+\eta,\ \phi'_\varepsilon(y_\eta) =
    \psi'_\varepsilon(y_\eta)>0.
  \end{cases}
\]
We have
\[
  \frac{\varepsilon}{2}\frac{\d}{\d y} \left( \phi'_\varepsilon \right)^2 =
  \frac{\d}{\d y}\prt{-\frac{1}{2} \phi_\varepsilon^2 + \frac{1}{4}
    \phi_\varepsilon^4},
\]
from which we deduce
\[
  \frac{\varepsilon}{2}(\phi'_\varepsilon)^2 = \frac{1}{4}\phi_\varepsilon^4 -
  \frac{1}{2}\phi_\varepsilon^2
  + C(\eta),
\]
with
\[\begin{split}
    C(\eta) &= -\frac{1}{4}(1+\eta)^4 +\frac{1}{2}(1+\eta)^2 +
    \frac{\varepsilon}{2}(\psi'_\varepsilon(y_\eta))^2\\
    &\gq -\frac{1}{4}(1+\eta)^4 +\frac{1}{2}(1+\eta)^2 +
    \frac{\varepsilon}{2}(\psi'_\varepsilon(y_0))^2,\\
  \end{split}
\]
where $C(\eta)$ is computed from the initial conditions at $y=y_\eta$ and $B_0 < y_0 < y_\eta$
is such that $\psi_\varepsilon(y_0)=1$.
Note that
$\eta\mapsto -\frac{1}{4}(1+\eta)^4 +\frac{1}{2}(1+\eta)^2 + \frac{\varepsilon}{2}(\psi'_\varepsilon(y_0))^2$
is decreasing on $\Rb_*^+$ and takes the value
$\frac{1}{4} + \frac{\varepsilon}{2}(\psi'_\varepsilon(y_0))^2 > \frac{1}{4}$
at $\eta = 0$. Thus, for $\eta>0$ small enough, $C(\eta) \gq \frac{1}{4}$ and
we have
\[
  \begin{split}
    \frac{\varepsilon}{2}(\phi_\varepsilon')^2
    &\gq \frac{1}{4}\phi_\varepsilon^4 - \frac{1}{2}\phi_\varepsilon^2 +
    \frac{1}{4} = \frac{1}{4}\prt{\phi_\varepsilon^2-1}^2,
  \end{split}
\]
hence
\[
  \phi_\varepsilon' \gq \frac{1}{\sqrt{2\varepsilon}}\prt{\phi_\varepsilon^2-1} \quad \mbox{on } [y_\eta,Y_\varepsilon).
\]
Finally, we consider the ODE
\[
  \xi'_{\varepsilon,\eta} =
  \frac{1}{\sqrt{2\varepsilon}}(\xi_{\varepsilon,\eta}^2 -1),\qquad
  \xi_{\varepsilon,\eta}(y_\eta) = 1+\eta,
\]
whose solution is
\[
  \xi_{\varepsilon,\eta}(y) = \frac{1+\frac{2}{\eta} +
    \exp\prt{\frac{y-y_\eta}{\sqrt{\varepsilon/2}}}}
  {1+\frac{2}{\eta} - \exp\prt{\frac{y-y_\eta}{\sqrt{\varepsilon/2}}}},
\]
which is defined only up to
$Y_{\varepsilon,\eta} \coloneqq \sqrt{\frac \varepsilon 2}\log\prt{1 + \frac{2}{\eta}} + y_\eta$
. Applying again Theorem~\ref{thm:comparison}, we have that
$\phi_\varepsilon(y) \gq \xi_{\varepsilon,\eta}(y)$ for any $y \gq y_\eta$
such that
both functions are still finite.
Putting everything together, we obtain that
$\psi_\varepsilon$ is only defined up to some $Y_\varepsilon$ with
$B_0 < Y_\varepsilon\lq Y_{\varepsilon,\eta}<\infty$ and that
\begin{equation}\label{eq:final_bound}
  \Forall y\in[y_\eta,Y_\varepsilon),\quad
  \psi_\varepsilon(y)\gq\xi_{\varepsilon,\eta}(y).
\end{equation}
These results are illustrated on Figure~\ref{fig:Xmu}, where we plotted the lower
bound $\xi_{\varepsilon,\eta}$ for $\varepsilon=0.1$ and $\eta=0.5$ as well as
a numerical approximation of $\psi_\varepsilon$.

\begin{figure}[h!]
  \centering
  \includegraphics[width=0.9\linewidth]{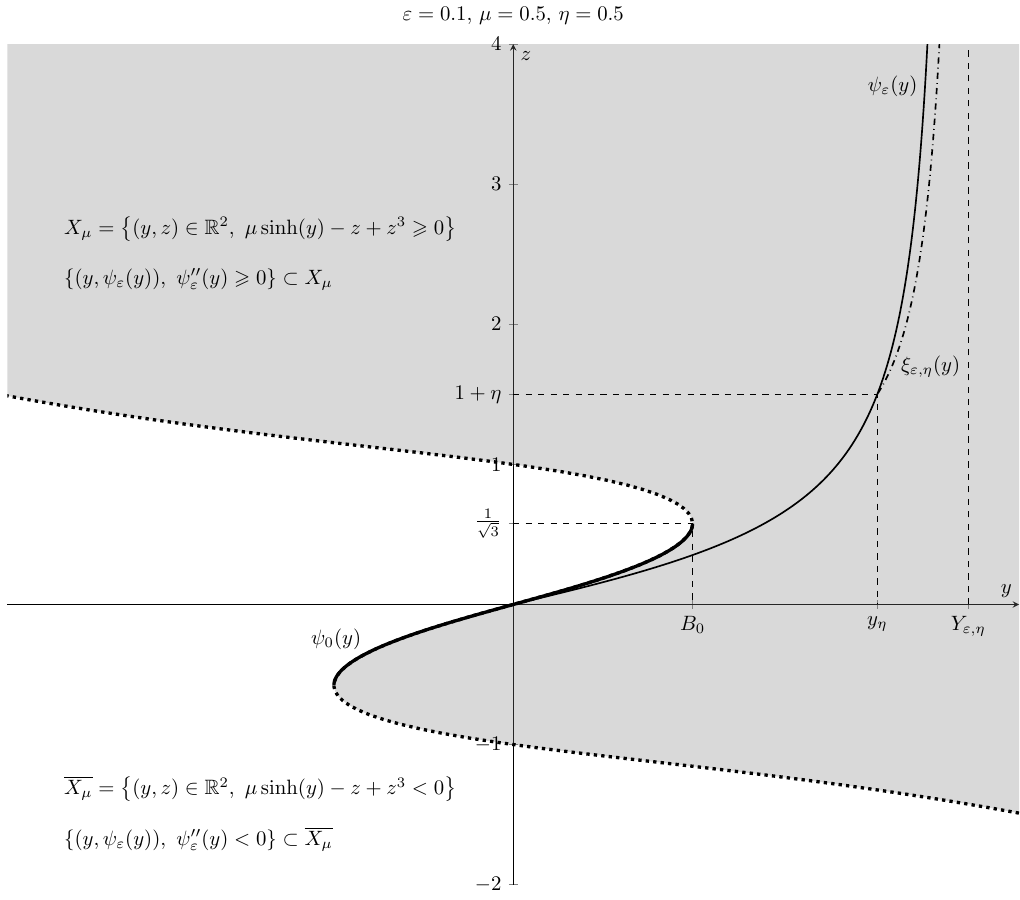}
  \caption{Description of $X_\mu$ for $\mu=0.5$, along with the plot of
    $\psi_\varepsilon$ and the lower bound $\xi_{\varepsilon,\eta}$ for
    $\varepsilon=0.1$, $\eta=0.5$. While $y < B_0$, $\psi_\varepsilon$ can
    possibly oscillate around $\psi_0$, but as soon as $y\gq B_0$,
    $\psi_\varepsilon$ is strictly convex and has no other choice than to explode in
    finite time $Y_\varepsilon \lq Y_{\varepsilon,\eta}$, where
    $Y_{\varepsilon,\eta}$ is the explosion time of the lower bound
    $\xi_{\varepsilon,\eta}$.  }
  \label{fig:Xmu}
\end{figure}

\bibliography{./biblio}

\begin{thebibliography}{10}

\bibitem{babuskaEigenvalueProblems1991}
I.~Babu{\v s}ka and J.~Osborn.
\newblock Eigenvalue problems.
\newblock In {\em Handbook of {{Numerical Analysis}}}, volume~2 of {\em Finite {{Element Methods}} ({{Part}} 1)}, pages 641--787. {Elsevier}, 1991.

\bibitem{benderAdvancedMathematicalMethods1999}
C.~M. Bender and S.~A. Orszag.
\newblock {\em Advanced Mathematical Methods for Scientists and Engineers. 1: {{Asymptotic}} Methods and Perturbation Theory}.
\newblock {Springer}, {New York Heidelberg}, 1999.

\bibitem{bernsteinNatureAnalytiqueSolutions1904}
S.~Bernstein.
\newblock {Sur la nature analytique des solutions des \'equations aux d\'eriv\'ees partielles du second ordre}.
\newblock {\em Mathematische Annalen}, 59(1-2):20--76, 1904.

\bibitem{blattAnalyticitySolutionsNonlinear2020}
S.~Blatt.
\newblock On the analyticity of solutions to non-linear elliptic partial differential systems.
\newblock {\em arXiv:2009.08762 [math.AP]}, 2020.

\bibitem{blochlProjectorAugmentedwaveMethod1994}
P.~E. Bl{\"o}chl.
\newblock Projector augmented-wave method.
\newblock {\em Physical Review B}, 50(24):17953--17979, 1994.

\bibitem{cancesNumericalAnalysisNonlinear2010}
E.~Canc{\`e}s, R.~Chakir, and Y.~Maday.
\newblock Numerical {{Analysis}} of {{Nonlinear Eigenvalue Problems}}.
\newblock {\em Journal of Scientific Computing}, 45(1):90--117, 2010.

\bibitem{cancesNumericalAnalysisplanewave2012}
E.~Canc{\`e}s, R.~Chakir, and Y.~Maday.
\newblock Numerical analysis of the planewave discretization of some orbital-free and {{Kohn-Sham}} models.
\newblock {\em ESAIM: Mathematical Modelling and Numerical Analysis}, 46(2):341--388, 2012.

\bibitem{friedmanRegularitySolutionsNonLinear1958}
A.~Friedman.
\newblock On the {{Regularity}} of the {{Solutions}} of {{Non-Linear Elliptic}} and {{Parabolic Systems}} of {{Partial Differential Equations}}.
\newblock {\em Indiana University Mathematics Journal}, 7(1):43--59, 1958.

\bibitem{giannozziQUANTUMESPRESSOModular2009}
P.~Giannozzi, S.~Baroni, N.~Bonini, M.~Calandra, R.~Car, C.~Cavazzoni, D.~Ceresoli, G.~L. Chiarotti, M.~Cococcioni, I.~Dabo, A.~Dal~Corso, S.~{de Gironcoli}, S.~Fabris, G.~Fratesi, R.~Gebauer, U.~Gerstmann, C.~Gougoussis, A.~Kokalj, M.~Lazzeri, L.~{Martin-Samos}, N.~Marzari, F.~Mauri, R.~Mazzarello, S.~Paolini, A.~Pasquarello, L.~Paulatto, C.~Sbraccia, S.~Scandolo, G.~Sclauzero, A.~P. Seitsonen, A.~Smogunov, P.~Umari, and R.~M. Wentzcovitch.
\newblock {{QUANTUM ESPRESSO}}: A modular and open-source software project for quantum simulations of materials.
\newblock {\em Journal of Physics. Condensed Matter: An Institute of Physics Journal}, 21(39):395502, 2009.

\bibitem{goedeckerSeparableDualspaceGaussian1996}
S.~Goedecker, M.~Teter, and J.~Hutter.
\newblock Separable dual-space {{Gaussian}} pseudopotentials.
\newblock {\em Physical Review B}, 54(3):1703, 1996.

\bibitem{gonzeAbinitProjectImpact2020}
X.~Gonze, B.~Amadon, G.~Antonius, F.~Arnardi, L.~Baguet, J.-M. Beuken, J.~Bieder, F.~Bottin, J.~Bouchet, E.~Bousquet, N.~Brouwer, F.~Bruneval, G.~Brunin, T.~Cavignac, J.-B. Charraud, W.~Chen, M.~C{\^o}t{\'e}, S.~Cottenier, J.~Denier, G.~Geneste, P.~Ghosez, M.~Giantomassi, Y.~Gillet, O.~Gingras, D.~R. Hamann, G.~Hautier, X.~He, N.~Helbig, N.~Holzwarth, Y.~Jia, F.~Jollet, W.~{Lafargue-Dit-Hauret}, K.~Lejaeghere, M.~A.~L. Marques, A.~Martin, C.~Martins, H.~P.~C. Miranda, F.~Naccarato, K.~Persson, G.~Petretto, V.~Planes, Y.~Pouillon, S.~Prokhorenko, F.~Ricci, G.-M. Rignanese, A.~H. Romero, M.~M. Schmitt, M.~Torrent, M.~J. {van Setten}, B.~Van~Troeye, M.~J. Verstraete, G.~Z{\'e}rah, and J.~W. Zwanziger.
\newblock The {{Abinit}} project: {{Impact}}, environment and recent developments.
\newblock {\em Computer Physics Communications}, 248:107042, 2020.

\bibitem{hartwigsenRelativisticSeparableDualspace1998}
C.~Hartwigsen, S.~Goedecker, and J.~Hutter.
\newblock Relativistic separable dual-space {{Gaussian}} pseudopotentials from {{H}} to {{Rn}}.
\newblock {\em Physical Review B}, 58(7):3641--3662, 1998.

\bibitem{hashimotoRemarkAnalyticitySolutions2006}
Y.~Hashimoto.
\newblock A {{Remark}} on the {{Analyticity}} of the {{Solutions}} for {{Non-Linear Elliptic Partial Differential Equations}}.
\newblock {\em Tokyo Journal of Mathematics}, 29(2):271--281, 2006.

\bibitem{herbstDFTKJulianApproach2021}
M.~F. Herbst, A.~Levitt, and E.~Canc{\`e}s.
\newblock {{DFTK}}: {{A Julian}} approach for simulating electrons in solids.
\newblock {\em Proceedings of the JuliaCon Conferences}, 3(26):69, 2021.

\bibitem{liuLimitedMemoryBFGS1989}
D.~C. Liu and J.~Nocedal.
\newblock On the limited memory {{BFGS}} method for large scale optimization.
\newblock {\em Mathematical Programming}, 45(1-3):503--528, Aug. 1989.

\bibitem{morreyAnalyticitySolutionsAnalytic1958}
C.~B. Morrey.
\newblock On the {{Analyticity}} of the {{Solutions}} of {{Analytic Non-Linear Elliptic Systems}} of {{Partial Differential Equations}}: {{Part I}}. {{Analyticity}} in the {{Interior}}.
\newblock {\em American Journal of Mathematics}, 80(1):198, 1958.

\bibitem{morreyAnalyticitySolutionsAnalytic1958b}
C.~B. Morrey.
\newblock On the {{Analyticity}} of the {{Solutions}} of {{Analytic Non-Linear Elliptic Systems}} of {{Partial Differential Equations}}: {{Part II}}. {{Analyticity}} at the {{Boundary}}.
\newblock {\em American Journal of Mathematics}, 80(1):219, 1958.

\bibitem{nottoliRobustOpensourceImplementation2023}
T.~Nottoli, I.~Giann{\`i}, A.~Levitt, and F.~Lipparini.
\newblock A robust, open-source implementation of the locally optimal block preconditioned conjugate gradient for large eigenvalue problems in quantum chemistry.
\newblock {\em Theoretical Chemistry Accounts}, 142(8):69, Aug. 2023.

\bibitem{perdewAccurateSimpleAnalytic1992}
J.~P. Perdew and Y.~Wang.
\newblock Accurate and simple analytic representation of the electron-gas correlation energy.
\newblock {\em Physical Review B}, 45(23):13244--13249, 1992.

\bibitem{petrovskiiAnalyticiteSolutionsSystemes1939}
I.~G. Petrovskii.
\newblock {Sur l'analyticit\'e des solutions des syst\`emes d'\'equations diff\'erentielles.}
\newblock {\em Matematiceskij sbornik}, 47(1):3--70, 1939.

\bibitem{ratcliffFlexibilitiesWaveletsComputational2020}
L.~E. Ratcliff, W.~Dawson, G.~Fisicaro, D.~Caliste, S.~Mohr, A.~Degomme, B.~Videau, V.~Cristiglio, M.~Stella, M.~D'Alessandro, S.~Goedecker, T.~Nakajima, T.~Deutsch, and L.~Genovese.
\newblock Flexibilities of wavelets as a computational basis set for large-scale electronic structure calculations.
\newblock {\em Journal of Chemical Physics}, 152(19):194110, 2020.

\bibitem{reedAnalysisOperators1978}
M.~Reed and B.~Simon.
\newblock {\em Analysis of Operators}.
\newblock Number~4 in Methods of {{Modern Mathematical Physics}}. {Academic Press}, 1978.

\bibitem{romeroABINITOverviewFocus2020}
A.~H. Romero, D.~C. Allan, B.~Amadon, G.~Antonius, T.~Applencourt, L.~Baguet, J.~Bieder, F.~Bottin, J.~Bouchet, E.~Bousquet, F.~Bruneval, G.~Brunin, D.~Caliste, M.~C{\^o}t{\'e}, J.~Denier, C.~Dreyer, P.~Ghosez, M.~Giantomassi, Y.~Gillet, O.~Gingras, D.~R. Hamann, G.~Hautier, F.~Jollet, G.~Jomard, A.~Martin, H.~P.~C. Miranda, F.~Naccarato, G.~Petretto, N.~A. Pike, V.~Planes, S.~Prokhorenko, T.~Rangel, F.~Ricci, G.-M. Rignanese, M.~Royo, M.~Stengel, M.~Torrent, M.~J. {van Setten}, B.~Van~Troeye, M.~J. Verstraete, J.~Wiktor, J.~W. Zwanziger, and X.~Gonze.
\newblock {{ABINIT}}: {{Overview}} and focus on selected capabilities.
\newblock {\em Journal of Chemical Physics}, 152(12):124102, 2020.

\bibitem{slaterSimplificationHartreeFockMethod1951}
J.~C. Slater.
\newblock A {{Simplification}} of the {{Hartree-Fock Method}}.
\newblock {\em Physical Review}, 81(3):385--390, 1951.

\bibitem{troullierEfficientPseudopotentialsPlanewave1991}
N.~Troullier and J.~L. Martins.
\newblock Efficient pseudopotentials for plane-wave calculations.
\newblock {\em Physical Review B}, 43(3):1993--2006, 1991.

\bibitem{vrabelDuffingTypeOscillatorBounded2013}
R.~Vrabel, P.~Tanuska, P.~Vazan, P.~Schreiber, and V.~Liska.
\newblock Duffing-{{Type Oscillator}} with a {{Bounded}} from above {{Potential}} in the {{Presence}} of {{Saddle-Center Bifurcation}} and {{Singular Perturbation}}: {{Frequency Control}}.
\newblock {\em Abstract and Applied Analysis}, 2013:1--7, 2013.

\bibitem{walterOrdinaryDifferentialEquations1998}
W.~Walter.
\newblock {\em Ordinary Differential Equations}.
\newblock Number 182 in Graduate Texts in Mathematics ; {{Readings}} in Mathematics. {Springer}, {New York}, 1998.

\bibitem{wazewskiSystemesEquationsInegualites1950}
T.~Wazewski.
\newblock Syst\`emes des \'equations et des in\'egualit\'es diff\'erentielles ordinaires aux deuxi\`emes membres monotones et leurs applications.
\newblock {\em Annales de la Soci\'et\'e polonaise de math\'ematique}, 23:112--166, 1950.

\end{thebibliography}
\nocite{*}
\bibliographystyle{abbrv}

\end{document}